\documentclass[12pt]{amsart}
\usepackage{amssymb}
\usepackage{tikz}
\allowdisplaybreaks
\newcommand{\XX}[1]{#1}
\newcommand{\LemmaOnFiniteGenerationOfRankTwoSteinbergGroups}{\XX{12.2}}
\newcommand{\TheoremOnFinitePresentationGeneralTools}{\XX{1.4}}
\newcommand{\PartOfThatTheoremDealingWithTwoSpherical}{(\XX{ii})}
\newcommand{\PartOfThatTheoremDealingWithFGRings}{(\XX{iii})}
\newcommand{\CorollaryOnFiniteGenerationOfKernel}{\XX{12.4}}
\newcommand{\TheoremOnPresentingPSt}{\XX{1.2}}
\newcommand{\PreSteinbergSection}{\XX{7}}
\newcommand{\SteinbergSection}{\XX{6}}
\newcommand{\RemarkOnOmittingTorusActionRelations}{\XX{7.13}}

\newcommand{\RelationsDefiningRootGroups}{\XX{(7.4)}}
\newcommand{\FirstActionOfSiSiOnSj}{\XX{(7.2)}}
\newcommand{\SecondActionOfSiSiOnSj}{\XX{(7.3)}}
\newcommand{\ActionOfSiSiOnXj}{\XX{(7.5)}}
\newcommand{\EqualityOfSiWithComplicatedWord}{\XX{(7.26)}}
\newcommand{\ArtinRelations}{\XX{(7.1)}}
\newcommand{\mEqualsTwoSandXinteraction}{\XX{(7.6)}}
\newcommand{\mEqualsThreeSandXinteraction}{\XX{(7.7)}}
\newcommand{\mEqualsFourSandXinteraction}{\XX{(7.8)}}
\newcommand{\mEqualsSixSandXinteraction}{\XX{(7.9)}}
\newcommand{\FirstTorusActionRelations}{\XX{(7.24)}}
\newcommand{\SecondTorusActionRelations}{\XX{(7.25)}}

\newcommand{\ChevAoneAone}{\XX{(7.10)}}
\newcommand{\FirstChevAtwo}{\XX{(7.11)}}
\newcommand{\LastChevAtwo}{\XX{(7.12)}}
\newcommand{\FirstChevBtwo}{\XX{(7.13)}}
\newcommand{\LastChevBtwo}{\XX{(7.16)}}
\newcommand{\FirstChevGtwo}{\XX{(7.17)}}
\newcommand{\LastChevGtwo}{\XX{(7.23)}}

\def\commutes{\!\rightleftarrows}
\def\singletonalpha{\{\alpha\}}

\newcommand{\N}{\mathbb{N}}     
\newcommand{\Z}{\mathbb{Z}}     
\newcommand{\Q}{\mathbb{Q}}     
\newcommand{\R}{\mathbb{R}}     
\newcommand{\C}{\mathbb{C}}     
\newcommand{\F}{\mathbb{F}}     
\newcommand{\SL}{{\rm SL}}      
\newcommand{\Sp}{{\rm Sp}}      
\newcommand{\St}{\mathfrak{St}} 
\newcommand{\StTits}{\mathfrak{St}^{\scriptstyle\rm Tits}} 
\newcommand{\PSt}{\mathfrak{PSt}} 
\newcommand{\CD}{\mathfrak{CD}} 

\newcommand{\Atilde}{\widetilde{A}}
\newcommand{\Btilde}{\widetilde{B}}
\newcommand{\Ctilde}{\widetilde{C}}
\newcommand{\Dtilde}{\widetilde{D}}
\newcommand{\Etilde}{\widetilde{E}}
\newcommand{\Ftilde}{\widetilde{F}}
\newcommand{\Gtilde}{\widetilde{G}}
\newcommand{\BCtilde}{\widetilde{BC}}
\newcommand{\Xtilde}{\widetilde{X}}
\newcommand{\even}{{\,\rm even}}
\newcommand{\odd}{{\,\rm odd}}
\newcommand{\zeromodthree}{{\rm\,0\,mod\,3}}

\newcommand{\alphabar}{\bar{\alpha}}
\newcommand{\betabar}{\bar{\beta}}
\newcommand{\gammabar}{\bar{\gamma}}
\newcommand{\deltabar}{\bar{\delta}}
\newcommand{\e}{\varepsilon}
\newcommand{\lambdabar}{\bar{\lambda}}
\newcommand{\mubar}{\bar{\mu}}
\newcommand{\sigmabar}{\bar{\sigma}}

\newcommand{\g}{\mathfrak{g}}   
\newcommand{\G}{\mathfrak{G}}   
\newcommand{\Runits}{R^*} 
\newcommand{\U}{\mathfrak{U}}   
\newcommand{\Phibar}{\overline{\Phi}} 
\newcommand{\Lambdabar}{\overline{\Lambda}} 
\newcommand{\htilde}{\tilde{h}} 
\newcommand{\stilde}{\tilde{s}} 
\newcommand{\sstar}{s^*}
\newcommand{\what}{\hat{w}} 
\newcommand{\sltwo}{\mathfrak{sl}_2}

\def\mathrlap#1{\mathchoice
{\rlap{$\displaystyle #1$}}%
{\rlap{$\textstyle #1$}}%
{\rlap{$\scriptstyle #1$}}%
{\rlap{$\scriptscriptstyle #1$}}}

\def\mathllap#1{\mathchoice
{\llap{$\displaystyle #1$}}%
{\llap{$\textstyle #1$}}%
{\llap{$\scriptstyle #1$}}%
{\llap{$\scriptscriptstyle #1$}}}

\let\iso\cong
\let\cong\equiv
\newcommand{\sset}{\subseteq}
\newcommand{\tensor}{\otimes}
\newcommand{\freeproduct}{\mathop{*}}
\newcommand{\rk}{\mathop{\rm rk}\nolimits}
\newcommand{\ad}{\mathop{\rm ad}\nolimits}

\newcommand{\Aut}{\mathop{\rm Aut}\nolimits}

\newtheorem{theorem}{Theorem}
\newtheorem{lemma}[theorem]{Lemma}
\newtheorem{corollary}[theorem]{Corollary}
\theoremstyle{remark}
\newtheorem*{remark}{Remark}

\numberwithin{equation}{section}
\numberwithin{theorem}{section}
\numberwithin{table}{section}

\begin{document}

\title{Presentation of affine Kac-Moody groups
  over rings}
\author{Daniel Allcock}
\thanks{Supported by NSF grant DMS-1101566}
\address{Department of Mathematics\\University of Texas, Austin}
\email{allcock@math.utexas.edu}
\urladdr{http://www.math.utexas.edu/\textasciitilde allcock}
\date{June 11, 2015}

\begin{abstract}
Tits has defined Steinberg groups and Kac-Moody groups for any root
system and any commutative ring~$R$.  We establish a Curtis-Tits style
presentation for the Steinberg group $\St$ of any rank${}\geq3$
irreducible affine root system, for any~$R$.  Namely, $\St$ is the
direct limit of the Steinberg groups coming from the $1$- and $2$-node
subdiagrams of the Dynkin diagram.  In fact we give a completely
explicit presentation.  Using this we show that $\St$ is finitely
presented if the rank is${}\geq4$ and $R$ is finitely generated as a
ring, or if the rank is~$3$ and $R$ is finitely generated as a module
over a subring generated by finitely many units.  Similar results hold
for the corresponding Kac-Moody groups when $R$ is a Dedekind domain
of arithmetic type.
\end{abstract}

\maketitle

\section{Introduction}
\label{sec-introduction}

\noindent
Suppose $R$ is a commutative ring and $A$ is one of the {\it
  ABCDEFG\/} Dynkin diagrams, or equivalently its Cartan matrix.
Steinberg defined what is now called the Steinberg group $\St_A(R)$,
by generators and relations \cite{Steinberg}.  It plays a central role
in K-theory and some aspects of Lie theory.


Kac-Moody algebras are infinite-dimensional generalizations of the
semisimple Lie algebras.  When $R=\R$ and $A$ is an affine Dynkin
diagram, the corresponding Kac-Moody group is a central extension of
the loop group of a finite-dimensional Lie group.  For a general ring
$R$ and any generalized Cartan matrix $A$, the definition of a
Kac-Moody group is due to Tits \cite{Tits}.  A difficulty in
tracing the story is that Tits began by defining a ``Steinberg group''
which unfortunately differs from Steinberg's original group when $A$
has an $A_1$ component.  This was resolved by
Morita-Rehmann \cite{Morita-Rehmann} by adding extra relations to
Tits' definition.  So there are two definitions of the Steinberg
group.  Increasing the chance of confusion, the definitions
agree for most $A$ of interest, including the irreducible affine
diagrams of rank${}\geq3$.  We follow Morita-Rehmann, so the
Steinberg group $\St_A(R)$ reduces to Steinberg's original group when
this is defined.  See section~\ref{sec-background} for further background on $\St$.

Tits then defined another functor $R\mapsto\tilde\G_{\!A}(R)$ as a
quotient of his version of the Steinberg group.  In this paper we will
omit the tilde and refer to $\G_{\!A}(R)$ as the Kac-Moody group of
type $A$ over $R$.  The  relations added by Morita-Rehmann to the definition of
$\St_A(R)$   are among the relations that Tits
imposed in his definition of $\G_{\!A}(R)$.  Therefore we may regard
$\G_{\!A}(R)$ as a quotient of $\St_A(R)$, just as Tits did, even
though our $\St_A(R)$ is not quite the same as his.  
See
section~\ref{sec-background} for further background on $\G$.

(Tits actually defined $\tilde\G_{\!D}(R)$ where $D$ is a root datum;
by $\G_{\!A}(R)$ we intend the root datum whose generalized Cartan
matrix is $A$ and which is ``simply-connected in the strong sense''
\cite[p.~551]{Tits}.  The general case differs from this one  by
enlarging or
shrinking the center of $\tilde\G_{\!D}(R)$.)

The meaning of ``Kac-Moody group'' is far from standardized.  In
\cite{Tits} Tits wrote down axioms (KMG1)--(KMG9) that one could
demand of a functor from rings to groups before calling it a Kac-Moody
functor.  He showed \cite[Thm.~$1'$]{Tits} that any such functor admits a
natural homomorphism from $\G_{\!A}$, which is an isomorphism at every
field.  So Kac-Moody groups over fields are well-defined, and over
general rings $\G_{\!A}$ approximates the yet-unknown ultimate
definition.  This is why we refer to $\G_{\!A}$ as the Kac-Moody
group.  But $\G_{\!A}$ does not quite satisfy Tits' axioms, so
ultimately some other language may be better.  See
section~\ref{sec-finite-presentations} for more remarks on this.

The purpose of this paper is to simplify Tits' presentations of
$\St_A(R)$ and $\G_{\!A}(R)$ when $A$ is an affine Dynkin
diagram of rank (number of nodes) at least~$3$.  
We will always take affine diagrams to be irreducible.
We will
show that 
$\St_A(R)$ and $\G_{\!A}(R)$ 
are finitely presented under quite weak
hypotheses on~$R$.  This is surprising because there is no obvious
reason for an infinite-dimensional group over (say) $\Z$ to be
finitely presented, and Tits' presentations are ``very'' infinite.
His generators are indexed by all pairs (root, ring element), and his
relations specify the commutators of many pairs of these
generators.  Subtle implicitly-defined coefficients appear throughout
his relations.

The main step in proving our finite presentation results is to first
establish smaller, and more explicit, presentations for $\St_A(R)$ and
$\G_{\!A}(R)$.  These presentations are not
necessarily finite, but they do apply to all~$R$.  In \cite{Allcock} we
wrote down a presentation for a group functor we called the
pre-Steinberg group $\PSt_A$.  We have reproduced it in
section~\ref{sec-the-presentation}, for any generalized Cartan matrix $A$.  The generators
are $S_i$ and $X_i(t)$ with $i$ varying over the nodes of the Dynkin
diagram and $t$ varying over $R$.  The relations are
\eqref{eq-additivity-in-simple-root-groups}--\eqref{eq-second-torus-action-relation}, but \eqref{eq-first-torus-action-relation}--\eqref{eq-second-torus-action-relation} may be omitted when $A$ is
$2$-spherical (it has no edges labeled~$\infty$)
and has no $A_1$ components.  This case includes all affine
diagrams of rank${}\geq3$.  The only way the presentation fails to be
finite is that the $X_i(t)$ and some of the relations are
parameterized by elements of $R$ (or pairs of elements).

The name ``pre-Steinberg group'' reflects the fact that there is a
natural map from $\PSt_A(R)$ to the Steinberg group $\St_A(R)$.  In
section~\ref{sec-background} we will describe this in a conceptual manner.  But in
terms of presentations it suffices to say that our $X_i(t)$ and $S_i$
map to the group elements 
$x_{\alpha_i}(t)$ and $\what_{\alpha_i}(1)$  in the Morita-Rehmann
definition of $\St_A(R)$ in \cite[\S2]{Morita-Rehmann}.  Our general
philosophy is that $\PSt_A(R)$ is interesting only as a means of
approaching $\St_A(R)$,  as in the following theorem, which is our main result.
  

\begin{theorem}[Presentation of affine Steinberg \& Kac-Moody groups]
\label{thm-presentations}
Suppose $A$ is an affine Dynkin diagram of rank${}\geq3$
and $R$ is a commutative ring.  Then the natural map
from the pre-Steinberg group $\PSt_A(R)$ to the Steinberg group
$\St_A(R)$ is an isomorphism.  In particular, $\St_A(R)$
has a presentation with generators $S_i$ and $X_i(t)$, with~$i$
varying over the simple roots and $t$ over $R$, and relations
\eqref{eq-additivity-in-simple-root-groups}--\eqref{eq-6-Chevalley-distant-short-and-long}.

One obtains Tits' Kac-Moody group $\G_{\!A}(R)$ by adjoining the relations
\begin{align}
\label{eq-torus-relations}
\htilde_i(u)\htilde_i(v)={}&\htilde_i(u v)
\\
\noalign{\noindent for all simple roots $i$ and all units $u,v$ of $R$, where}
\notag
\htilde_i(u)
:={}&\stilde_i(u)\stilde_i(-1).
\\
\notag
\stilde_i(u)
:={}&X_i(u)S_i X_i(1/u) S_i^{-1} X_i(u).
\end{align}
\end{theorem}

We remark that if $A$ is a spherical diagram (that is, its Weyl group
is finite) then it follows immediately from an alternate description
of $\PSt_A$ that $\PSt_A\to\St_A$ is an isomorphism; see
section~\ref{sec-background} or \cite[\S\PreSteinbergSection]{Allcock}.  So
theorem~\ref{thm-presentations} extends the isomorphism
$\PSt_A\iso\St_A$ from the spherical case to the affine case,
except for the two affine diagrams of rank~$2$.  See
\cite{Allcock-Carbone} for a further extension, to the simply-laced
hyperbolic case.

\smallskip
For a moment we return to the case where $A$ is an arbitrary
generalized Cartan matrix.  If $B_1\sset B_2$ are two subdiagrams of
$A$ then there is a natural homomorphism
$\PSt_{B_1}(R)\to\PSt_{B_2}(R)$.  This is because the generators and
relations of $\PSt_{B_1}(R)$ are among those of $\PSt_{B_2}(R)$, by
the fact that our presentations of these groups are defined in terms
of the nodes and edges of these subdiagrams of $A$.  Using these maps,
we consider the directed system of groups $\PSt_B(R)$ where $B$ varies
over the subdiagrams of $A$ of rank${}\leq2$.  It is a formality that
the direct limit is $\PSt_A(R)$; this is just an abstract way of
saying that each generator or relation of $\PSt_A(R)$ already appears
in the presentation of some $\PSt_B(R)$ with $B$ of rank${}\leq2$.

When $A$ is affine of rank${}\geq3$, $\PSt_A(R)\to\St_A(R)$ is an
isomorphism by theorem~\ref{thm-presentations}.  And $\PSt_B(R)\to\St_B(R)$ is an
isomorphism for every proper subdiagram $B$ of $A$, since such
subdiagrams are spherical.  It follows that we may replace $\PSt$ by
$\St$ throughout the preceding paragraph, proving the following
result.  The point is that affine Steinberg groups of rank${}\geq3$
are built up from the classical Steinberg groups of types $A_1$,
$A_1^2$, $A_2$, $B_2$ and~$G_2$.

\begin{corollary}[Curtis-Tits presentation]
\label{cor-Steinberg-as-direct-limit}
Suppose $A$ is an affine Dyn\-kin diagram of rank${}\geq3$
and $R$ is a commutative ring.
Then $\St_A(R)$ is the direct limit of the groups $\St_B(R)$, where $B$
varies over the subdiagrams of $A$ of rank${}\leq2$, and the maps
between these groups are as specified above.
The same result also holds with $\St$ replaced by $\G$
throughout. 
\qed
\end{corollary}

An informal way to restate
corollary~\ref{cor-Steinberg-as-direct-limit} is that a presentation
for $\St_A(R)$ can be got by  amalgamating one's favorite
presentations for the $\St_B(R)$'s.
Splitthoff \cite{Splitthoff} discovered quite weak sufficient
conditions for the latter groups to be finitely presented.  When these
hold, one would therefore expect $\St_A(R)$ also to be finitely
presented.  The next theorem expresses this idea precisely.
Claim~\ref{item-rk-3-and-module-finite-over-suitable-subring} is part of \cite[Thm.\ \TheoremOnFinitePresentationGeneralTools]{Allcock}.  See section~\ref{sec-finite-presentations} for
the proof of claim~\ref{item-rk-4-and-finitely-generated-ring}.

\begin{theorem}[Finite presentability]
\label{thm-finite-presentation-of-Steinberg}
Suppose $A$ is an affine Dynkin diagram
and $R$ is any commutative ring.
Then the Steinberg group $\St_A(R)$ is fi\-ni\-tely presented as a group if either
\begin{enumerate}
\item
\label{item-rk-4-and-finitely-generated-ring}
$\rk A>3$ and $R$ is finitely generated as a ring, or
\item
\label{item-rk-3-and-module-finite-over-suitable-subring}
$\rk A=3$ and $R$ is finitely generated as a module over a subring
generated by finitely many units.
\end{enumerate}
In either case, if the unit group of $R$ is finitely generated as an
abelian group, then Tits' Kac-Moody group $\G_{\!A}(R)$ is finitely presented as a group.
\end{theorem}

One of the main motivations for Splitthoff's work was to understand
when the Chevalley-Demazure groups, over Dedekind
domains of interest in number theory, are finitely presented.  This
was finally settled by Behr
\cite{Behr-number-field-case}\cite{Behr-function-field-case},
capping a long series of works by many authors.  The following
analogue of these results follows immediately from
theorem~\ref{thm-finite-presentation-of-Steinberg}.  How close the
analogy is depends on how well $\G_{\!A}$ approximates whatever plays
the role of the Chevalley-Demazure group scheme in the setting of
Kac-Moody theory.

\begin{corollary}[Finite presentation in arithmetic contexts]
\label{cor-Kac-Moody-finite-presentation}
Suppose $K$ is a global field, meaning a finite extension of $\Q$ or
$\F_q(t)$.  Suppose $S$ is a nonempty finite set of places of $K$, including
all  infinite places in the number field case.  Let $R$ be the ring of
$S$-integers in $K$.

Suppose $A$ is an  affine 
Dynkin diagram.  Then Tits' Kac-Moody group $\G_{\!A}(R)$ is finitely
presented if 
\begin{enumerate}
\item
\label{item-rk-bigger-than-3-and-function-field-and-S=2}
$\rk A>3$ when $K$ is a function field and $|S|=1$;
\item
\label{item-rk-3-otherwise}
$\rk A\geq3$ otherwise.
\qed
\end{enumerate}
\end{corollary}

We  remark that if $R$ is a field then the $\G_{\!A}$ case of
corollary~\ref{cor-Steinberg-as-direct-limit} is due to Abramenko-M\"uhlherr
\cite{Abramenko-Muhlherr}\cite{Devillers-Muhlherr}.  Namely, suppose $A$ is any
generalized Cartan matrix which is $2$-spherical, and that $R$ is a field (but not
$\F_2$ if $A$ has a double bond, and neither $\F_2$ nor $\F_3$ if $A$
has a multiple bond).  Then
$\G_{\!A}(R)$ is the direct limit of the groups $\G_{\!B}(R)$.  
Abramenko-M\"uhlherr
\cite[p.~702]{Abramenko-Muhlherr} state that if $A$ is
affine then one can remove the restrictions $R\neq\F_2,\F_3$.

One of our goals in this work is to bring Kac-Moody groups into the
world of geometric and combinatorial group theory, which mostly
addresses finitely presented groups.  For example: which Kac-Moody
groups admit classifying spaces with finitely many cells below some
chosen dimension?  What other finiteness properties do they have?  Do
they have Kazhdan's property~$T$? What isoperimetric inequalities do
they satisfy in various dimensions?  Are there (non-split) Kac-Moody
groups over local fields whose uniform lattices (suitably defined) are
word hyperbolic?  Are some Kac-Moody groups (or classes of them)
quasi-isometrically rigid?  We find the last question very attractive,
since the corresponding answer
\cite{Eskin-Farb}\cite{Farb-Schwartz}\cite{Kleiner-Leeb}\cite{Schwartz}
for lattices in Lie groups is deep.

Regarding property $T$ we would like to mention work of
Hartnick-K\"ohl \cite{Hartnick-Kohl}, who show that many Kac-Moody
groups over local fields have property~$T$ when equipped with the
Kac-Peterson topology. Also, Shalom \cite{Shalom} and Neuhauser
\cite{Neuhauser} respectively showed that the loop groups of (i.e., the spaces of continuous
maps from $S^1$ to) $\SL_n(\C)$ and $\Sp_{2n}(\C)$ have property~$T$.

The author is very grateful to the Japan Society for the Promotion of
Science and to Kyoto University, for their support and hospitality.
He would also like to thank Lisa Carbone and the referees for very
helpful comments on earlier versions of the paper.

\section{Presentation of the pre-Steinberg group $\PSt_A(R)$}
\label{sec-the-presentation}

\noindent
Suppose $R$ is any commutative ring and $A$ is any generalized Cartan
matrix.  Write $I$ for the set of $A$'s nodes, and for $i,j\in I$
write $m_{i j}$ for the order of the product of the corresponding
generators of the Weyl group.  
Following
\cite[\S\PreSteinbergSection]{Allcock} the pre-Steinberg group $\PSt_A(R)$ is defined by
the following presentation.  The generators are $S_i$ and $X_i(t)$
with $t\in R$.  The relations are \eqref{eq-additivity-in-simple-root-groups}--\eqref{eq-second-torus-action-relation} below, in which
$i,j$ vary over $I$ and $t,u$ vary over~$R$.  We use the notation
$Y\commutes Z$ to say that $Y$ and $Z$ commute.

If $A$ has no $A_1$ components and is $2$-spherical (all $m_{i j}$'s
are finite), then the last two relations \eqref{eq-first-torus-action-relation}--\eqref{eq-second-torus-action-relation} follow
from the others and may be omitted \cite[Remark~\RemarkOnOmittingTorusActionRelations]{Allcock}.  If $A$ is
affine of rank${}\geq3$ then it satisfies this condition, and
our main result (theorem~\ref{thm-presentations}) is
that the presentation equally well defines the Steinberg group $\St_A(R)$.

\def\FOOleft{130pt}
\def\FOOright{150pt}
\def\FOOequals#1#2{
\kern\FOOleft\mathllap{#1}&{}=\mathrlap{#2}\kern\FOOright
}
\def\FOOcommutes#1#2{
\kern\FOOleft\mathllap{#1}&{}\commutes\mathrlap{#2}\kern\FOOright
}

For every $i\in I$ we impose the relations
\begin{align}
\label{eq-additivity-in-simple-root-groups}
\FOOequals{X_i(t)X_i(u)}{X_i(t+u)}
\\
\label{eq-equality-of-S-i-with-complicated-word}
\FOOequals{S_i}{X_i(1) S_i X_i(1) S_i^{-1} X_i(1)}
\\
\noalign{For all $i,j$ we impose the relations}
\label{eq-action-of-S-i-2-on-S-j}
\FOOequals{S_i^2 S_j S_i^{-2}}{S_j^{(-1)^{A_{i j}}}}
\\
\label{eq-action-of-S-i-2-on-X-j}
\FOOequals{S_i^2 X_j(t) S_i^{-2}}{X_j\bigl((-1)^{A_{i j}}t\bigr)}
\\
\noalign{Whenever $m_{i j}=2$ we impose the relations}
\label{eq-2-Artin-relation}
\FOOequals{S_i S_j}{S_j S_i}
\\
\label{eq-2-interaction-of-S-i-and-X-j}
\FOOcommutes{S_{i}}{X_{j}(t)}
\\
\label{eq-2-Chevalley-relation}
\FOOcommutes{X_{i}(t)}{X_{j}(u)}
\\
\noalign{Whenever $m_{i j}=3$ we impose the relations}
\label{eq-3-Artin-relation}
\FOOequals{S_i S_j S_i}{S_j S_i S_j}
\\
\label{eq-3-interaction-of-S-j-S-i-and-X-j-X-i}
\FOOequals{S_{j} S_{i} X_{j}(t)}{X_{i}(t) S_{j} S_{i}}
\\
\label{eq-3-Chevalley-close-roots}
\FOOcommutes{X_{i}(t)}{S_{i} X_{j}(u) S_{i}^{-1}}
\\
\label{eq-3-Chevalley-distant-roots}
\FOOequals{[X_{i}(t),X_{j}(u)]}{S_{i} X_{j}(t u) S_{i}^{-1}}
\end{align}

Whenever $m_{i j}=4$ we impose the following relations; in
\eqref{eq-4-Chevalley-close-short-and-long}--\eqref{eq-4-Chevalley-distant-short-and-long},
$s$ resp.\ $l$ refers to whichever of $i$ and $j$ is the shorter
resp.\ longer root.
\begin{align}
\label{eq-4-Artin-relation}
\FOOequals{S_i S_j S_i S_j}{S_j S_i S_j S_i}
\\
\label{eq-4-S-i-S-j-S-i-commutes-with-X-j}
\FOOcommutes{S_{i} S_{j} S_{i}}{X_{j}(t)}
\\
\label{eq-4-Chevalley-close-short-and-long}
\FOOcommutes{S_{s} X_{l}(t) S_{s}^{-1}}{S_{l} X_{s}(u) S_{l}^{-1}}
\\
\label{eq-4-Chevalley-orthogonal-long}
\FOOcommutes{X_{l}(t)}{S_{s} X_{l}(u) S_{s}^{-1}}
\\
\label{eq-4-Chevalley-orthogonal-short}
\FOOequals{[X_{s}(t),S_{l} X_{s}(u) S_{l}^{-1}]}{S_{s} X_{l}(-2 t u)
  S_{s}^{-1}}
\\
\label{eq-4-Chevalley-distant-short-and-long}
\FOOequals{[X_{s}(t),X_{l}(u)]}{S_{l} X_{s}(-t u) S_{l}^{-1}\cdot S_{s} X_{l}(t^{2} u) S_{s}^{-1}}
\end{align}

Whenever $m_{i j}=6$ we impose the following relations; $s$ and $l$
have the same meaning they had in the previous paragraph.
\begin{align}
\label{eq-6-Artin-relation}
\FOOequals{S_i S_j S_i S_j S_i S_j}{S_j S_i S_j S_i S_j S_i}
\\
\label{eq-6-S-i-S-j-S-i-S-j-S-i-commutes-with-X-j}
\FOOcommutes{S_{i} S_{j} S_{i} S_{j} S_{i}}{X_{j}(t)}
\\
\label{eq-6-Chevalley-close-long}
\FOOcommutes{X_{l}(t)}{S_{l} S_{s} X_{l}(u) S_{s}^{-1} S_{l}^{-1}}
\\
\label{eq-6-Chevalley-adjacent-short-and-long}
\FOOcommutes{S_{s} S_{l} X_{s}(t) S_{l}^{-1} S_{s}^{-1}}{S_{l} S_{s}
  X_{l}(u) S_{s}^{-1} S_{l}^{-1}}
\\
\label{eq-6-Chevalley-orthogonal-short-and-long}
\FOOcommutes{S_{s} X_{l}(t) S_{s}^{-1}}{S_{l} X_{s}(u) S_{l}^{-1}}
\\
\label{eq-6-Chevalley-distant-long}
\FOOequals{[X_{l}(t),S_{s} X_{l}(u) S_{s}^{-1}]}{S_{l} S_{s} X_{l}(t
  u) S_{s}^{-1} S_{l}^{-1}}
\\
\label{eq-6-Chevalley-close-short}
\FOOequals{[X_{s}(t),S_{s} S_{l} X_{s}(u) S_{l}^{-1}
    S_{s}^{-1}]}{S_{s} X_{l}(3 t u) S_{s}^{-1}}
\\
\label{eq-6-Chevalley-distant-short}
\FOOequals{[X_{s}(t),S_{l} X_{s}(u) S_{l}^{-1}]}{S_{s} S_{l} X_{s}(-2 t u) S_{l}^{-1} S_{s}^{-1}\cdot{}}
\\
\notag
{}\cdot S_{s} X_{l}(-3 &t^{2} u) S_{s}^{-1} 
\cdot
S_{l} S_{s} X_{l}(-3 t u^{2}) S_{s}^{-1} S_{l}^{-1}
\\
\label{eq-6-Chevalley-distant-short-and-long}
\FOOequals{[X_{s}(t),X_{l}(u)]}{ S_{s} S_{l} X_{s}(t^{2} u) S_{l}^{-1}
  S_{s}^{-1}\cdot{} }
\\
\notag
{}\cdot S_{l} X_{s}(-t u) S_{l}^{-1} \cdot S_{s} X_{l}(&t^{3} u) S_{s}^{-1} \cdot
S_{l} S_{s} X_{l}(-t^{3} u^{2}) S_{s}^{-1} S_{l}^{-1}
\end{align}

Officially, the next two relations \eqref{eq-first-torus-action-relation}--\eqref{eq-second-torus-action-relation} are part of the
presentation of $\PSt_A(R)$.  But as mentioned above, they may be
omitted if $A$ is $2$-spherical without $A_1$ components.  We let $r$
vary over the units of $R$ and impose the relations
\begin{align}
\label{eq-first-torus-action-relation}
\FOOequals{\htilde_{i}(r) X_{j}(t)
  \htilde_{i}(r)^{-1}}{X_{j}\bigl(r^{A_{i j}} t\bigr)}
\\
\label{eq-second-torus-action-relation}
\FOOequals{\htilde_{i}(r) \, S_{j} X_{j}(t) S_{j}^{-1} \,
  \htilde_{i}(r)^{-1}}{
S_{j} X_{j}\bigl(r^{-A_{i j}} t\bigr) S_{j}^{-1}}
\end{align}
where $\htilde_i(r)$ was defined in theorem~\ref{thm-presentations}.

Because we have organized the relations differently than we did in
\cite{Allcock}, we will state the correspondence explicitly:
\eqref{eq-additivity-in-simple-root-groups}=\cite[\RelationsDefiningRootGroups]{Allcock}. 
\eqref{eq-equality-of-S-i-with-complicated-word}=\cite[\EqualityOfSiWithComplicatedWord]{Allcock}.
\eqref{eq-action-of-S-i-2-on-S-j}=\cite[\FirstActionOfSiSiOnSj--\SecondActionOfSiSiOnSj]{Allcock}.
\eqref{eq-action-of-S-i-2-on-X-j}=\cite[\ActionOfSiSiOnXj]{Allcock}. 
\eqref{eq-2-Artin-relation}$\,\cup\,$\eqref{eq-3-Artin-relation}$\,\cup\,$\eqref{eq-4-Artin-relation}$\,\cup\,$\discretionary{}{}{}\eqref{eq-6-Artin-relation}=\cite[\ArtinRelations]{Allcock}.
\eqref{eq-2-interaction-of-S-i-and-X-j}=\cite[\mEqualsTwoSandXinteraction]{Allcock}.
\eqref{eq-2-Chevalley-relation}=\cite[\ChevAoneAone]{Allcock}, the $A_1^2$ Chevalley relation.
\eqref{eq-3-interaction-of-S-j-S-i-and-X-j-X-i}=\cite[\mEqualsThreeSandXinteraction]{Allcock}.
\eqref{eq-3-Chevalley-close-roots}--\eqref{eq-3-Chevalley-distant-roots}=\cite[\FirstChevAtwo--\LastChevAtwo]{Allcock}, the
$A_2$ Chevalley relations.
\eqref{eq-4-S-i-S-j-S-i-commutes-with-X-j}=\cite[\mEqualsFourSandXinteraction]{Allcock}.
\eqref{eq-4-Chevalley-close-short-and-long}--\eqref{eq-4-Chevalley-distant-short-and-long}=\cite[\FirstChevBtwo--\LastChevBtwo]{Allcock}, the
$B_2$ Chevalley relations.
\eqref{eq-6-S-i-S-j-S-i-S-j-S-i-commutes-with-X-j}=\cite[\mEqualsSixSandXinteraction]{Allcock}.
\eqref{eq-6-Chevalley-close-long}--\eqref{eq-6-Chevalley-distant-short-and-long}=\cite[\FirstChevGtwo--\LastChevGtwo]{Allcock}, the
$G_2$ Chevalley relations.
\eqref{eq-first-torus-action-relation}=\cite[\FirstTorusActionRelations]{Allcock}.
\eqref{eq-second-torus-action-relation}=\cite[\SecondTorusActionRelations]{Allcock}.


\section{Steinberg and pre-Steinberg Groups}
\label{sec-background}

\noindent
Our goal in this section is to describe the Steinberg group and to
give a second description of the pre-Steinberg group.  This
description makes visible its natural map to the Steinberg group, and
is the form we will use for our calculations in section~\ref{sec-isomorphism-for-affine-groups}.

We work in the setting of \cite{Tits} and \cite{Allcock}, so $R$ is a
commutative ring and $A$ is a generalized Cartan matrix.  This matrix
determines a complex Lie algebra $\g$ called the Kac-Moody algebra,
and we write $\Phi$ for the set of real roots of $\g$.  For each real
root $\alpha$, its root space $\g_\alpha$ comes with a distinguished
pair of (complex vector space) generators, each the negative of the other.  We
write $\g_{\alpha,\Z}$ for their integral span, and define the root
group $\U_\alpha$ as $\g_{\alpha,\Z}\tensor R\iso R$.  
Tits'
definition of the Steinberg group begins with the free
product $\freeproduct_{\alpha\in\Phi}\,\U_\alpha$. 

We emphasize that there is no natural way to choose an isomorphism
$R\to\U_{\alpha}$.  If $\{\pm e\}$ are the two distinguished
generators for $\g_\alpha$, then there are two natural choices for the
parameterization of $\U_\alpha$, namely $t\mapsto(\pm e)\tensor t$.
Often we will choose one of these and call it $X_{\alpha}$; we speak of this
as a ``sign choice''.  Making such a choice sometimes simplifies
computations, but sometimes it is better to treat both possibilities with
equal respect.

In Tits' definition of $\St_A(R)$, the relations have the following form.
He calls a pair $\alpha,\beta\in\Phi$ {\it prenilpotent\/} if some
element of the Weyl group $W$ sends both $\alpha,\beta$ to positive
roots, and some other element of $W$ sends both to negative roots.  A
consequence of this condition is that every root in $\N\alpha+\N\beta$
is real, which enabled Tits to write down Chevalley-style relators for
$\alpha,\beta$.  That is, for every prenilpotent pair $\alpha,\beta$
he imposes relations of the form
\begin{equation}
\label{eq-form-of-Chevalley-relations}
\bigl[\hbox{element of $\U_\alpha$},\hbox{element of $\U_\beta$}\bigr]
=
\prod_{\gamma\in\theta(\alpha,\beta)-\{\alpha,\beta\}}
\!\!\!\!\!\!(\hbox{element of $\U_\gamma$})
\end{equation}
where $\theta(\alpha,\beta):=(\N\alpha+\N\beta)\cap\Phi$ and
$\N=\{0,1,2,\dots\}$.  The exact relations are given in a rather
implicit form in \cite[\S3.6]{Tits}.  Writing them down explicitly
requires choosing parameterizations of 
$\U_\alpha$, $\U_\beta$ and each~$\U_\gamma$.  
We suppose this has been done as above, with  the parameterizations
being  $X_\alpha$, $X_\beta$ and
the various $X_\gamma$.  Then the relations
take the form
\begin{equation}
\label{eq-explicit-form-of-Chevalley-relations}
\bigl[X_\alpha(t),X_\beta(u)]
=
\prod_{
\displaystyle\mathop{\scriptstyle{\rm
    roots}\,\,\gamma=m\alpha+n\beta}_{{\rm with}\,\,m,n\geq1}
}
\!\!\!\!\!\!
X_\gamma\bigl(N_{\alpha\beta\gamma}\,t^m u^n\bigr)
\end{equation}
where the $N_{\alpha\beta\gamma}$ are integers determined by the
structure constants of $\g$, the sign choices made in parameterizing
the root groups, and the ordering of the terms on the right side.  See
\cite[\S\S3.4--3.6]{Tits} for details, or section~\ref{sec-isomorphism-for-affine-groups} for the cases we will
need.  Morita showed that the right side has at most~$1$ term
except when $(\Q\alpha\oplus\Q\beta)\cap\Phi$ has type $B_2$ or $G_2$,
and found simple formulas for the constants (up to sign).  See
\cite{Morita-commutator-relations} and \cite{Morita-root-strings}.

For Tits, this is the end of the definition of the Steinberg group.
We called this group $\StTits_A(R)$ in \cite{Allcock}, to avoid
confusion with $\St_A(R)$ itself, which we take to also satisfy the
Morita-Rehmann relations.  These extra relations play the role of
making the ``maximal torus'' and ``Weyl group'' in $\StTits_A(R)$ act
in the expected way on root spaces.  These relations follow from the
Chevalley relations when $A$ is $2$-spherical without $A_1$
components, so the reader could skip down to the definition of
$\G_{\!A}(R)$.

Here is a terse description of the Morita-Rehmann relations; see
\cite[relations (B$'$)]{Morita-Rehmann} or \cite[\S\SteinbergSection]{Allcock} for more details.  For each simple root
$\alpha\in\Phi$ and each of the two choices $e$ for a generator of
$\g_{\alpha,\Z}$, we impose relations as follows. 
By a standard
construction, the choice of $e$ distinguishes a generator $f$ for
$\g_{-\alpha,\Z}$.  Using $e$ and $f$ as above, we obtain
parameterizations of $\U_\alpha$ and $\U_{-\alpha}$ which we will call
$X_e$ and $X_f$.  For $r\in\Runits$ we define
$\stilde_e(r)=X_e(r)X_f(1/r)X_e(r)$ and
$\htilde_e(r)=\stilde_e(r)\stilde_e(-1)$.  Morita and Rehmann impose
relations that describe the actions of $\stilde_e(1)$ and
$\htilde_e(r)$ on every $\U_\beta$, where $\beta$ varies over $\Phi$.
First, conjugation by $\stilde_e(1)$ sends $\U_\beta$ to
$\U_{s_\alpha(\beta)}$ in the same way that
$\sstar_e:=(\exp\ad_e)(\exp\ad_f)(\exp\ad_e)\in\Aut\g$ does.  (Here
$s_\alpha$ is the reflection in $\alpha$, and for the relation to make
sense one must check that $\sstar_e$ sends $\g_{\beta,\Z}$ to
$\g_{s_\alpha(\beta),\Z}$.)  Second, every $\htilde_e(r)$ acts on
$\U_\beta\iso R$ by scaling by $r^{\langle\alpha^\vee,\beta\rangle}$,
where $\alpha^\vee$ is the coroot associated to~$\alpha$.

The quotient of $\StTits_A(R)$ by all these relations is the
definition of the Steinberg group $\St_A(R)$, and agrees with
Steinberg's original group when $A$ is spherical.  We remark that we
let $e$ vary over both possible choices of generator for
$\g_{\alpha,\Z}$ just to avoid choosing one.  But one could choose one
without harm, because it turns out that the relations imposed for $e$
are the same as those imposed for $-e$.  Also, Morita and Rehmann
write $\what_\alpha$ rather than $\stilde_e$, and their definition of
it uses $X_f(-1/r)$ rather than $X_f(1/r)$.  This sign merely reflects
the fact that they use a different sign on $f$ than Tits does, in the
``standard'' basis $e,f,h$ for $\sltwo$.

The Kac-Moody group $\G_{\!A}(R)$ is defined as the quotient of
$\St_A(R)$ by the relations \eqref{eq-torus-relations}.

\medskip
In section~\ref{sec-the-presentation} we defined the pre-Steinberg group $\PSt_A(R)$ in
terms of generators and relations.  But it also has an ``intrinsic''
definition: the same as $\St_A(R)$, except that Tits'
Chevalley relations are imposed only for
{\it classically nilpotent} pairs $\alpha,\beta$.  This means that
$(\Q\alpha+\Q\beta)\cap\Phi$ is finite
and $\alpha+\beta\neq0$.  This is equivalent to
$\alpha,\beta$ satisfying $\alpha+\beta\neq0$ and lying in some
$A_1$, $A_1^2$, $A_2$, $B_2$ or $G_2$ root system.  As the name suggests,
such a pair is prenilpotent.  So $\PSt_A(R)$ is defined the same way
as $\St_A(R)$, just omitting the Chevalley relations for prenilpotent
pairs that are not classically prenilpotent.  In particular,
$\St_A(R)$ is a quotient of $\PSt_A(R)$, hence the prefix ``pre-''.

In \cite{Allcock} we defined $\PSt_A(R)$ this way, and then showed
that it is has the presentation in
section~\ref{sec-the-presentation}.  In this paper, for ease of exposition we defined
$\PSt_A(R)$ 
by this presentation.  But we will use the above ``intrinsic''
description in the proof of theorem~\ref{thm-presentations}.  So equality between the
two versions of $\PSt_A(R)$ is essential for our work.  We proved this
in \cite[Thm.\ \TheoremOnPresentingPSt]{Allcock}, which we restate as follows:

\begin{theorem}[The two models of $\PSt_A(R)$]
\label{thm-two-models-of-pre-Steinberg}
Let $A$ be a generalized Cartan matrix and $R$ a commutative ring.
For each simple root $\alpha_i$, choose one of the two distinguished
parameterizations $X_{e_i}:R\to\U_{\alpha_i}$.  Then the pre-Steinberg
group as defined in section~\ref{sec-the-presentation} is isomorphic to the pre-Steinberg
group as defined above, by $S_i\mapsto\stilde_{e_i}(1)$ and
$X_i(t)\mapsto X_{e_i}(t)$.
\qed
\end{theorem}

\section{Nomenclature for affine root systems}
\label{sec-new-nomenclature}

\noindent
Our proof of theorem~\ref{thm-presentations}, appearing in the next
section, refers to the root system as a whole, with the simple roots
playing no special role.  It is natural in this setting to use a
nomenclature for the affine root systems that emphasizes this global
perspective.  Our notation in table~\ref{tab-root-system-names} is
close to that in Moody-Pianzola \cite[\S3.5]{Moody-Pianzola}.  The
differences are that our superscripts describe the construction
of the root systems, and that we use a tilde to indicate affineness.
For the affine root systems obtained by ``folding'', Kac'
nomenclature \cite[pp.~54--55]{Kac} emphasizes not the affine root
system itself but rather the one being folded.

\begin{table}
\begin{tabular}{lllll}
&\cite{Moody-Pianzola}
&\cite{Kac}
&condition
\\
$\Atilde{}_{n}$
&$A_n^{(1)}$
&$A_{n}^{(1)}$
&$n\geq1$
\\
$\Btilde{}_{n}$
&$B_n^{(1)}$
&$B_{n}^{(1)}$
&$n\geq2$
\\
$\Ctilde{}_{n}$
&$C_n^{(1)}$
&$C_{n}^{(1)}$
&$n\geq2$
\\
$\Dtilde{}_{n}$
&$D_n^{(1)}$
&$D_{n}^{(1)}$
&$n\geq3$
\\
$\Etilde{}_{n}$
&$E_n^{(1)}$
&$E_{n}^{(1)}$
&$n=6,7,8$
\\
$\Ftilde{}_{\!4}$
&$F_4^{(1)}$
&$F_{4}^{(1)}$
\\
$\Gtilde{}_2$
&$G_2^{(1)}$
&$G_2^{(1)}$
\\
$\Btilde{}_n^\even$
&$B_n^{(2)}$
&$D_{n+1}^{(2)}$
&$n\geq2$
\\
$\Ctilde{}_n^\even$
&$C_n^{(2)}$
&$A_{2n-1}^{(2)}$\ \ \ 
&$n\geq2$
\\
$\BCtilde{}_n^\odd$
&$BC_n^{(2)}$\ \ \ 
&$A_{2n}^{(2)}$
&$n\geq1$
\\
$\Ftilde{}_{\!4}^\even$
&$F_4^{(2)}$
&$E_6^{(2)}$
\\
$\Gtilde{}_2^\zeromodthree$\ \ \ 
&$G_2^{(3)}$
&$D_4^{(3)}$
\end{tabular}
\smallskip
\caption{Our and others' names for affine root systems; see
  section~\ref{sec-new-nomenclature}.}
\label{tab-root-system-names}
\end{table}

It is very easy to describe the set $\Phi$ of real roots in the root
system $\Xtilde_n^{\cdots}$.  Let $\Phibar$ be a root system of type
$X_n$, let $\Lambdabar$ be its root lattice, and let $\Lambda$ be
$\Lambdabar\oplus\Z$.  Then $\Phi\sset\Lambda$ is the set of pairs
$(\hbox{root of $X_n$},\discretionary{}{}{} m\in\Z)$ satisfying the
condition that if the root is long then $m$ has the property
``$\cdots$'' indicated in the superscript, if any.

A set of simple roots can be described as follows.  We begin with a
set of simple roots for the root system $\Phi_0\sset\Phi$ consisting
of roots of the form
$(\alphabar,0)$.  This is an $X_n$ root system except for
$\BCtilde{}_n^\odd$, when it has type $B_n$.  The last simple
root is $(\alphabar,1)$, where $\alphabar$ is the lowest root of
$\Phi_0$ in the
absence of a superscript, or twice the lowest short root for
$\BCtilde{}_n^\odd$, or the lowest short root in all other cases.
This can be used to verify the correspondences between our
nomenclature and those of Kac and Moody.

The condition on $n$ in table~\ref{tab-root-system-names} is the weakest condition for
which the definition of $\Xtilde{}_n^{\cdots}$ makes sense.  If one wishes to avoid
duplication, so that each isomorphism class of affine root system
appears exactly once, then one should omit one of
$\Atilde{}_3\iso\Dtilde{}_3$, one of $\Btilde{}_2\iso\Ctilde{}_2$ and
one of $\Btilde{}_2^\even\iso\Ctilde{}_2^\even$.  Both \cite{Kac} and
\cite{Moody-Pianzola} omit $\Dtilde{}_3$, $\Btilde{}_2$ and
$\Ctilde{}_2^\even$.  Also, \cite{Moody-Pianzola} gives $A_1^{(2)}$ as an
alternate name for $BC_1^{(2)}$.


\section{The isomorphism $\PSt_A(R)\to\St_A(R)$}
\label{sec-isomorphism-for-affine-groups}

\noindent 
This section is devoted to proving theorem~\ref{thm-presentations},
whose hypotheses we assume throughout.  
In light of theorem~\ref{thm-two-models-of-pre-Steinberg},
our goal is to show that the
Chevalley relations for the classically prenilpotent pairs imply those
of the remaining prenilpotent pairs.  We will begin by saying which
pairs of real roots are prenilpotent and which are classically
prenilpotent.  Then we will analyze the pairs that are prenilpotent
but not classically prenilpotent.

We fix the affine Dynkin diagram $A$, write
$\Phi,\Phibar,\Lambda,\Lambdabar$ as in
section~\ref{sec-new-nomenclature}, and use an overbar to indicate
projections of roots from $\Phi$ to $\Phibar$.  
It is easy to see that
$\alpha,\beta\in\Phi$ are classically prenilpotent just if they are
equal or their
projections $\alphabar,\betabar\in\Phibar$ are linearly independent.
The following lemma describes which pairs of roots are prenilpotent
but not classically prenilpotent, and what their Chevalley relations
are (except for one special case discussed later).

\begin{lemma}
\label{lem-prenilpotent-but-not-classically-prenilpotent-pairs}
The following are
equivalent:
\begin{enumerate}
\item
\label{item-prenilpotent-but-not-classically-prenilpotent}
$\alpha,\beta$ are prenilpotent but
not classically prenilpotent;
\item
\label{item-projections-differ-by-positive factor}
$\alpha\neq\beta$ and 
$\alphabar,\betabar$ differ by a
positive scalar factor;
\item
\label{item-projections-are-equal-except-for-special-case}
$\alpha\neq\beta$, and either
$\alphabar,\betabar$ are equal, or else one is twice the other and
  $\Phi=\BCtilde{}_n^\odd$.
\end{enumerate}
When these equivalent conditions hold, the Chevalley relations between
$\U_\alpha,\U_\beta$ are
$[\U_\alpha,\U_\beta]=1$, unless $\Phi=\BCtilde{}_n^\odd$, $\alphabar$
and $\betabar$ are the same short root of $\Phibar=BC_n$, and
$\alpha+\beta\in\Phi$.
\end{lemma}

\begin{proof}
We think of the Weyl group
$W$ acting on affine space in the usual way, with each
root corresponding to an open halfspace.  A root is positive if its
halfspace contains the fundamental chamber, or negative if not.
Recall that two roots $\alpha,\beta\in\Phi$ form a prenilpotent pair
if some element $w_+$ of $W$ sends both to
positive roots, and some  $w_-\in W$ sends both to negative
roots.  The existence of both $w_{\pm}$ is equivalent to: some chamber
lies in the halfspaces of both $\alpha$ and $\beta$, and some other
chamber lies in neither of them.  (Proof: apply $w_{\pm}$ to the
fundamental chamber rather than to $\{\alpha,\beta\}$.)  By Euclidean
geometry, this happens just if: either their bounding hyperplanes are
non-parallel, or else their bounding hyperplanes are parallel and
one halfspace contains the other.  In the first case $\alphabar$
and $\betabar$ are linearly independent, so $\alpha$, $\beta$ are
classically prenilpotent.  In the second case, $\alphabar$ and
$\betabar$ differ by a positive scalar.  If $\alpha$ and $\beta$ are
equal then they form a classically prenilpotent pair. Otherwise
they do not, because $(\Q\alpha\oplus\Q\beta)\cap\Phi$ is infinite. 
This proves the equivalence
of \ref{item-prenilpotent-but-not-classically-prenilpotent} and \ref{item-projections-differ-by-positive factor}.

To see the
equivalence of \ref{item-projections-differ-by-positive factor} and
\ref{item-projections-are-equal-except-for-special-case} we refer to
the fact that $\Phibar$ is a reduced root system (i.e, the only
positive multiple of a root that can be a root is that root itself)
except in the case $\Phi=\BCtilde{}_n^\odd$.  In this last case, the
only way one root of $\Phibar=BC_n$ can be a positive multiple of a
different root is that the long roots are got by doubling the short roots.

The proof of the final claim is similar.  Except in the excluded case,
we have $\Phibar\cap(\N\alphabar+\N\betabar)=\{\alphabar,\betabar\}$.  The
corresponding claim for $\Phi$ follows, so
$\theta(\alpha,\beta)-\{\alpha,\beta\}$ is empty and the right hand side of
\eqref{eq-explicit-form-of-Chevalley-relations} is the identity.  That is, the Chevalley relations for $\alpha,\beta$
read $[\U_\alpha,\U_\beta]=1$.
(In the excluded case we remark that
$\Phi\cap(\N\alpha+\N\beta)=\{\alpha,\beta,\alpha+\beta\}$.  So the Chevalley relations
set the commutators of elements of $\U_\alpha$ with elements of
$\U_\beta$ equal to  certain elements of $\U_{\alpha+\beta}$. See
case~\XX{6} below.)
\end{proof}

Recall from theorem~\ref{thm-two-models-of-pre-Steinberg} that
$\St_A(R)$ may be got from $\PSt_A(R)$ by adjoining the Chevalley
relations for every prenilpotent pair $\alpha,\beta$ that is not
classically prenilpotent.  So to prove theorem~\ref{thm-presentations}
it suffices to show that that these relations already hold in
$\PSt:=\PSt_A(R)$.  
In light of lemma~\ref{lem-prenilpotent-but-not-classically-prenilpotent-pairs},
the proof falls into seven cases, according to
$\Phi$ and the relative position of $\alphabar$ and $\betabar$.
Conceptually, they are organized as follows; see below for their exact
hypotheses.  Case~\XX{1} applies if $\alphabar=\betabar$ is a long
root of some $A_2$ root system in $\Phibar$.  Case~\XX{2}
(\hbox{resp.\ \XX{3}}) applies if $\alphabar=\betabar$ is a long
(resp.\ short) root of some $B_2$ root system in $\Phibar$.
Case~\XX{4} applies if $\alphabar=\betabar$ is a short root of
$\Phibar=G_2$.  The rest of the cases are specific to
$\Phi=\BCtilde{}_n^\odd$.  Case~\XX{5} applies if
$\betabar=2\alphabar$.  Case \XX{6} or~\XX{7} applies if
$\alphabar=\betabar$ is a short root of $BC_n$.  There are two cases
because $\alpha+\beta$ might or might not be a root.

In every
case but one we must establish $[\U_\alpha,\U_\beta]=1$.  Each case
begins by choosing two roots in $\Phi$, of which $\beta$ is a
specified linear combination, and whose projections to $\Phibar$ are
specified.  Given the global description of $\Phi$ from
section~\ref{sec-new-nomenclature}, this is always easy.  Then we use the Chevalley
relations for various classically prenilpotent pairs to deduce the
Chevalley relations for $\alpha,\beta$. 

\begin{proof}[Case~\XX{1} of theorem~\ref{thm-presentations}]
{\it Assume $\alphabar=\betabar$ is a root of $\Phibar=A_{n\geq2}$,
  $D_n$ or $E_n$, or a long root of $\Phibar=G_2$.}  Choose
$\gammabar,\deltabar\in\Phibar$ as shown, and choose lifts
$\gamma,\delta\in\Phi$ summing to $\beta$.  (Choose any
$\gamma\in\Phi$ lying over $\gammabar$, define $\delta=\beta-\gamma$,
and use the global description of $\Phi$ to check that $\delta\in\Phi$.
This is trivial except in the case $\Phi=\Gtilde_2^\zeromodthree$,
when it is easy.)
$$
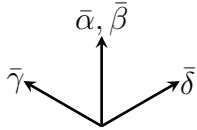
\begin{tikzpicture}[>=stealth, line width=1pt, scale=1.2]
\draw[->] (0,0) -- (-.866,.5) node {$\gammabar$\,\,\,};
\draw[->] (0,0) -- (0,1) node {\raise15pt\hbox{$\alphabar,\betabar$}};
\draw[->] (0,0) -- (.866,.5) node {\,\,\,$\deltabar$};
\end{tikzpicture}
$$
Because $\alphabar+\gammabar,\alphabar+\deltabar\notin\Phibar$, it
follows that $\alpha+\gamma,\alpha+\delta\notin\Phi$.  So
  the Chevalley relations
  $[\U_\alpha,\U_\gamma]=[\U_\alpha,\U_\delta]=1$ hold.  The Chevalley
  relations for $\gamma,\delta$ imply
  $[\U_\gamma,\U_\delta]=\U_{\gamma+\delta}=\U_\beta$.
(These relations are 
\eqref{eq-6-Chevalley-distant-long}
in the $G_2$ case and 
\eqref{eq-3-Chevalley-distant-roots}
in the others.  One can write
them as $[X_\gamma(t),X_{\delta}(u)]=X_{\gamma+\delta}(t u)$ in the
notation of the next paragraph.)
Since $\U_\alpha$ commutes with $\U_\gamma$ and $\U_\delta$, it 
commutes with the group they generate, hence $\U_\beta$.
\end{proof}

The other cases use the same strategy: express an element of
$\U_\beta$ in terms of other root groups, and then evaluate its
commutator with an element of $\U_\alpha$.  But the calculations are
more delicate.  We will work with explicit elements
$X_\gamma(t)\in\U_\gamma$ for various roots $\gamma\in\Phi$.  Here~$t$
varies over $R$, and the definition of $X_\gamma(t)$ depends on
choosing a basis vector $e_\gamma$ for the corresponding root space
$\g_\gamma\sset\g$ as explained in section~\ref{sec-background}. 
For each $\gamma$ there are two
possibilities for $e_\gamma$.  The point of making
these sign choices is to write down the relations explicitly.

For
example, if $s,l\in I$ are the short and long roots of a $B_2$
subdiagram of $A$, then 
we copy their relations from \eqref{eq-4-Chevalley-distant-short-and-long}:
\begin{equation}
\label{eq-quoted-Chevalley-Relation-s-and-l}
[X_s(t),X_l(u)]
=
S_l X_s(-t u) S_l^{-1}
\cdot
S_s X_l(t^{2} u) S_s^{-1}
\end{equation}
for all $t,u\in R$.  
The reason for writing the right side this way is
to avoid making choices: to write down the relation, one only needs to
specify generators $e_s$ and $e_l$ for $\g_s$ and
$\g_l$, not the other root spaces involved.  But for explicit
computation one must choose generators for these other root spaces.
Because $S_s$ and $S_l$ permute the root spaces in the same way  the
reflections in $s$ and $l$ do, the terms on
the right of
\eqref{eq-quoted-Chevalley-Relation-s-and-l}
lie in $\U_{l+s}$ and $\U_{l+2s}$.
Therefore, after choosing suitable generators $e_{l+s}$
and $e_{l+2s}$ for $\g_{l+s}$ and
$\g_{l+2s}$, we may
rewrite~\eqref{eq-quoted-Chevalley-Relation-s-and-l}
as
\begin{equation}
\label{eq-Chevalley-135-degrees-s-and-l}
[X_s(t),X_l(u)]
=
X_{l+s}(-t u)
\cdot 
X_{l+2s}(t^2u)
\end{equation}
Now, if $\sigma$ and $\lambda$ are short and long simple roots for
{\it any} copy of $B_2$ in $\Phi$, then some element $w$ of the Weyl
group sends some pair of simple roots to them.  Taking $s$ and $l$ to
be this pair, and  defining $X_\sigma$, $X_\lambda$,
$X_{\lambda+\sigma}$ and $X_{\lambda+2\sigma}$ as the $w$-conjugates
of $X_s$, $X_l$, $X_{l+s}$ and $X_{l+2s}$, we can write  the
Chevalley relation for $\sigma$ and $\lambda$ by 
applying the substitution $s\mapsto\sigma$ and $l\mapsto\lambda$ to \eqref{eq-Chevalley-135-degrees-s-and-l}:
\begin{equation}
\label{eq-Chevalley-135-degrees-sigma-and-lambda}
[X_\sigma(t),X_\lambda(u)]
=
X_{\lambda+\sigma}(-t u)
\cdot 
X_{\lambda+2\sigma}(t^2u)
\end{equation}
In this way we can obtain the Chevalley relations we will need, for
any classically prenilpotent pair, from the ones listed explicitly in
section~\ref{sec-the-presentation}.  One could also refer to any other
standard reference, for example \cite[\S5.2]{Carter}.

The root system $\BCtilde{}_{n\geq2}^\odd$ appears as a possibility in 
several cases, including the next one.  We
will use ``short'', ``middling'' and ``long'' to refer
to its three different root lengths.

\begin{proof}[Case~\XX{2} of theorem~\ref{thm-presentations}]
{\it Assume $\alphabar=\betabar$ is a long root of
  $\Phibar=B_{n\geq2}$, $C_{n\geq2}$, $BC_{n\geq2}$ or $F_4$.}  
Our first step is to 
choose roots $\lambdabar,\sigmabar\in\Phibar$ as pictured:
$$
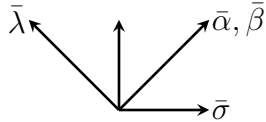
\begin{tikzpicture}[>=stealth, line width=1pt, scale=1.2]
\draw[->] (0,0) -- (-1,1) node {\llap{$\lambdabar$}\,};
\draw[->] (0,0) -- (0,1);
\draw[->] (0,0) -- (1,1) node {\,\rlap{$\alphabar,\betabar$}};
\draw[->] (0,0) -- (1,0) node {\,\rlap{$\sigmabar$}};
\end{tikzpicture}
$$ This is easily done using any standard description of $\Phibar$.
(Note: although $\lambdabar$ stands for ``long'' and $\sigmabar$ for ``short'', 
$\sigmabar$ is actually a middling root in the case $\Phibar=BC_n$.)  

Our
second step is to choose lifts $\lambda,\sigma\in\Phi$ of them with
$\beta=\lambda+2\sigma$.  If $\Phi=\Btilde{}_n$, $\Ctilde{}_n$ or
$\Ftilde{}_4$ then one chooses any lift $\sigma$ of $\sigmabar$ and
defines $\lambda$ as $\beta-2\sigma$.  This works since every element
of $\Lambda$ lying over a root of $\Phibar$ is a root of $\Phi$.  If
$\Phi=\Btilde{}_n^\even$, $\Ctilde{}_n^\even$, $\Ftilde{}_4^\even$ or
$\BCtilde{}_n^\odd$ then this argument might fail since $\Phi$ is
``missing'' some long roots.  Instead, one chooses any
$\lambda\in\Phi$ lying over $\lambdabar$ and defines $\sigma$ as
$(\beta-\lambda)/2$.  Now, $\beta-\lambda=(\betabar-\lambdabar,m)$
with $m$ being even by the meaning of the superscript
${}^\even$ or ${}^\odd$.  Also, $\betabar-\lambdabar$ is divisible
by~$2$ in $\Lambdabar$ by the figure above.  It follows that
$\sigma\in\Lambda$.  Then, as an element
of $\Lambda$ lying over a short (or middling) root of $\Phibar$, $\sigma$ lies
in~$\Phi$.

Because $\sigma,\lambda$ are simple roots for a $B_2$ root system
inside $\Phi$, their Chevalley relation \eqref{eq-Chevalley-135-degrees-sigma-and-lambda}
holds in $\PSt$.  This shows that 
any element of
$\U_\beta=\U_{\lambda+2\sigma}$ can be written in the form
\begin{equation}
\label{eq-getting-Chevalley-relation-for-2-long-roots-for-affine-B2}
(\hbox{some }x_{\lambda+\sigma}\in\U_{\lambda+\sigma})
\cdot
\bigl[
(\hbox{some }x_\sigma\in\U_\sigma),
(\hbox{some }x_\lambda\in\U_\lambda)
\bigr]
.
\end{equation}
Referring to the picture of $\Phibar$ shows that
$\alpha+\lambda+\sigma\notin\Phi$.  Therefore the Chevalley relations in
$\PSt$ include $[\U_\alpha,\U_{\lambda+\sigma}]=1$.  In particular,
$\U_\alpha$ commutes with the first term of
\eqref{eq-getting-Chevalley-relation-for-2-long-roots-for-affine-B2}.
The same argument shows that $\U_\alpha$ also commutes with the other
terms, hence with any element of $\U_\beta$. 
This shows that the Chevalley relations present in $\PSt$ imply
$[\U_\alpha,\U_\beta]=1$, as desired.
\end{proof}

\begin{proof}[Case~\XX{3} of theorem~\ref{thm-presentations}]
{\it Assume $\alphabar=\betabar$ is a short root of $\Phibar=B_{n\geq2}$,
  $C_{n\geq2}$ or $F_4$, or a middling root of $\Phibar=BC_{n\geq2}$.}
We may choose $\lambda,\sigma\in\Phi$ with sum $\beta$ and the following
projections to $\Phibar$ (by a simpler argument than in the previous case):
$$
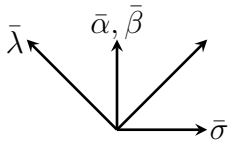
\begin{tikzpicture}[>=stealth, line width=1pt, scale=1.2]
\draw[->] (0,0) -- (-1,1) node {\llap{$\lambdabar$}\,};
\draw[->] (0,0) -- (0,1) node {\raise15pt\hbox{$\alphabar,\betabar$}};
\draw[->] (0,0) -- (1,1);
\draw[->] (0,0) -- (1,0) node {\,\rlap{$\sigmabar$}};
\end{tikzpicture}
$$ The Chevalley relations for $\sigma,\lambda$ are
\eqref{eq-Chevalley-135-degrees-sigma-and-lambda}, showing that any element of
$\U_\beta=\U_{\sigma+\lambda}$ can be written in the form
\begin{equation}
\label{eq-getting-Chevalley-relation-for-2-short-roots-for-affine-B2}
\bigl[
(\hbox{some }x_\sigma\in\U_\sigma),
(\hbox{some }x_\lambda\in\U_\lambda)
\bigr]
\cdot
(\hbox{some }x_{\lambda+2\sigma}\in\U_{\lambda+2\sigma})
.
\end{equation}
As in the previous case, we will conjugate this by an arbitrary
element of $\U_\alpha$.  This requires the following Chevalley relations.
We have
$[\U_\alpha,\U_\lambda]=1$ 
and 
$[\U_\alpha,\U_{\lambda+2\sigma}]=1$ 
by
the same argument as before.  What is new is that the Chevalley relations for
$\alpha,\sigma$ depend on whether $\alpha+\sigma$ is a root.  If it is,
then we get $[\U_\alpha,\U_\sigma]\sset\U_{\alpha+\sigma}$, and if not then we
get $[\U_\alpha,\U_\sigma]=1$.  In the second case we see that $\U_\alpha$
commutes with~\eqref{eq-getting-Chevalley-relation-for-2-short-roots-for-affine-B2},
proving $[\U_\alpha,\U_\beta]=1$ and therefore finishing the proof.

In the first case, conjugating \eqref{eq-getting-Chevalley-relation-for-2-short-roots-for-affine-B2} by a element of $\U_\alpha$ yields
$$
\bigl[
x_\sigma\cdot(\hbox{some }x_{\alpha+\sigma}\in\U_{\alpha+\sigma}),
x_\lambda
\bigr]
\cdot
x_{\lambda+2\sigma}
$$ which we can simplify by further use of Chevalley relations.
Namely, neither $\lambda+\alpha+\sigma$ nor $\alpha+2\sigma$ is a
root, so $\U_{\alpha+\sigma}$ centralizes $\U_\lambda$ and
$\U_\sigma$.  So $x_{\alpha+\sigma}$ centralizes the other terms in
the commutator, hence drops out, leaving
\eqref{eq-getting-Chevalley-relation-for-2-short-roots-for-affine-B2}.
This shows that conjugation by any element of $\U_\alpha$ leaves
invariant every element of $\U_\beta$.  That is,  $[\U_\alpha,\U_\beta]=1$.
\end{proof}

\begin{proof}[Case~\XX{4} of theorem~\ref{thm-presentations}]
{\it Assume $\alphabar=\betabar$ is a short root of $\Phibar=G_2$.}
This is the hardest case by far.  Begin by choosing roots
$\sigmabar,\lambdabar\in\Phibar$ as shown, with lifts
$\sigma,\lambda\in\Phi$ summing to $\beta$.
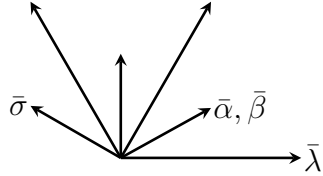
$$
\begin{tikzpicture}[>=stealth, line width=1pt, scale=1.2]
\draw[->] (0,0) -- (-1,.577) node {\llap{$\sigmabar$}\,};
\draw[->] (0,0) -- (-1,1.732);
\draw[->] (0,0) -- (0,1.155);
\draw[->] (0,0) -- (1,1.732);
\draw[->] (0,0) -- (1,.554) node {\,\rlap{$\alphabar,\betabar$}};
\draw[->] (0,0) -- (2,0) node {\,\rlap{$\lambdabar$}};
\end{tikzpicture}
$$
Many different root groups appear in the argument, so we choose a
generator $e_\gamma$ of $\gamma$'s root space, for each  $\gamma\in\Phi$ which is a
nonnegative linear combination of $\alpha,\sigma,\lambda$.

Next we write down the $G_2$ Chevalley relations in $\PSt$ that we
will need, derived from
\eqref{eq-6-Chevalley-close-long}--\eqref{eq-6-Chevalley-distant-short-and-long}.
We will write them down in the $\Phi=\Gtilde{}_2$ case and then
comment on the simplifications that occur if
$\Phi=\Gtilde{}_2^\zeromodthree$.  After negating some of the
$e_\gamma$, for $\gamma$ involving $\sigma$ and $\lambda$ but not
$\alpha$, we may suppose that the Chevalley relations
\eqref{eq-6-Chevalley-distant-short-and-long} for $\sigma,\lambda$
read
\begin{equation}
\label{eq-Chevalley-relation-for-s-and-l-for-G2-short-root-case}
\begin{split}
[X_\sigma(t),&X_\lambda(u)]
=
\\
&X_{2\sigma+\lambda}(t^2u)
X_{\sigma+\lambda}(-t u)
X_{3\sigma+\lambda}(t^3u)
X_{3\sigma+2\lambda}(-t^3u^2).
\end{split}
\end{equation}
Then we may negate $e_{\alpha+2\sigma+\lambda}$ if necessary, to suppose the Chevalley
relations \eqref{eq-6-Chevalley-close-short}
for $\alpha,2\sigma+\lambda$ read
\begin{equation}
\label{eq-Chevalley-relations-for-a-and-2s+l-for-G2-short-root-case}
[X_\alpha(t),X_{2\sigma+\lambda}(u)]
=
X_{\alpha+2\sigma+\lambda}(3t u).
\end{equation}
After negating some of the $e_\gamma$ for $\gamma$ involving $\alpha$ and $\sigma$ but not
$\lambda$, we may suppose that the Chevalley relations 
\eqref{eq-6-Chevalley-distant-short} for $\sigma$ and
$\alpha$ read
\begin{equation}
\label{eq-Chevalley-relations-for-a-and-s-for-G2-short-root-case}
[X_\sigma(t),X_\alpha(u)]
=
X_{\alpha+\sigma}(-2t u)
X_{\alpha+2\sigma}(-3t^2 u)
X_{2\alpha+\sigma}(-3t u^2)
\end{equation}
We know the Chevalley relations
\eqref{eq-6-Chevalley-close-short}
for $\sigma$ and $\alpha+\sigma$ have the
form
\begin{equation}
\label{eq-Chevalley-relations-for-s-and-a+s-for-G2-short-root-case}
[X_\sigma(t),X_{\alpha+\sigma}(u)]
=
X_{\alpha+2\sigma}(3\e t u)
\end{equation}
where $\e=\pm1$.  We cannot choose the sign because we've already used
our freedom to negate $e_{\alpha+2\sigma}$ in order to get
\eqref{eq-Chevalley-relations-for-a-and-s-for-G2-short-root-case}.  
Similarly, we know that the Chevalley relations 
\eqref{eq-6-Chevalley-distant-long}
for $\lambda$ and $\alpha+2\sigma$ are
\begin{equation}
\label{eq-Chevalley-relations-for-l-and-a+2s-for-G2-short-roots-case}
[X_\lambda(t),X_{\alpha+2\sigma}(u)]
=
X_{\alpha+2\sigma+\lambda}(\e't u)
\end{equation}
for some $\e'=\pm1$.  (We will see at the very end that $\e=\e'=1$.)

We were able to write down these relations because we could work out
the roots in the positive span of any two given roots.  This used the
assumption $\Phi=\Gtilde{}_2$, but now suppose
$\Phi=\Gtilde{}_2^\zeromodthree$.  It may happen that some of the
vectors appearing in the previous paragraph, projecting to long roots
of~$\Phibar=G_2$, are not roots of~$\Phi$.  One can check that if
$\alpha-\beta$ is divisible by~$3$ in $\Lambda$ then there is no
change.  On the other hand, if $\alpha-\beta\not\cong0$ mod~$3$ then
$\alpha+2\sigma+\lambda$, $\alpha+2\sigma$ and $2\alpha+\sigma$ are
not roots.  Because
$\bigl(\Q\alpha\oplus\Q(2\sigma+\lambda)\bigr)\cap\Phi$ now has type
$A_2$ rather than $G_2$,
\eqref{eq-Chevalley-relations-for-a-and-2s+l-for-G2-short-root-case}
is replaced by $[\U_\alpha,\U_{2\sigma+\lambda}]=1$, from
\eqref{eq-3-Chevalley-close-roots}.  And
$(\Q\alpha\oplus\Q\sigma)\cap\Phi$ also has type~$A_2$ now, so
\eqref{eq-Chevalley-relations-for-a-and-s-for-G2-short-root-case} is
replaced by $[X_\sigma(t),X_\alpha(t)]=X_{\alpha+\sigma}(t u)$,
obtained from \eqref{eq-3-Chevalley-distant-roots}, and
\eqref{eq-Chevalley-relations-for-s-and-a+s-for-G2-short-root-case} is
replaced by $[\U_\sigma,\U_{\alpha+\sigma}]=1$, from
\eqref{eq-3-Chevalley-close-roots}.  Finally, there is no relation
\eqref{eq-Chevalley-relations-for-l-and-a+2s-for-G2-short-roots-case}
because there is no longer a root group $\U_{\alpha+2\sigma}$.  The
calculations below use the relations
\eqref{eq-Chevalley-relation-for-s-and-l-for-G2-short-root-case}--\eqref{eq-Chevalley-relations-for-l-and-a+2s-for-G2-short-roots-case}.
To complete the proof, one must also carry out a similar calculation
using \eqref{eq-Chevalley-relation-for-s-and-l-for-G2-short-root-case}
and the altered versions of
\eqref{eq-Chevalley-relations-for-a-and-2s+l-for-G2-short-root-case}--\eqref{eq-Chevalley-relations-for-s-and-a+s-for-G2-short-root-case}.
This calculation is so much easier that we omit it.

The long roots 
$3\sigma+2\lambda$, 
$\alpha+2\sigma+\lambda$ and $2\alpha+\sigma$ all  lie over
$3\sigmabar+2\lambdabar$.  These root groups commute with all others
that will appear, by the Chevalley relations in $\PSt$, and they
commute with each other by case~\XX{1} above.  We will use this
without specific mention.

Since $\beta=\sigma+\lambda$, we may take \eqref{eq-Chevalley-relation-for-s-and-l-for-G2-short-root-case}  with $t=1$ and
rearrange, to express any element of $\U_\beta$ as
\begin{equation}
\label{eq-G2-expression-for-element-of-U-beta}
X_\beta(u)
=
X_{3\sigma+\lambda}(u)
X_{3\sigma+2\lambda}(-u^2)
[X_\lambda(u),X_\sigma(1)]
X_{2\sigma+\lambda}(u).
\end{equation}
We use this to express the commutators generating
$[\U_\alpha,\U_\beta]$:
\begin{equation}
\label{eq-G2-big-mess-for-commutator-1}
\begin{split}
[X_\alpha(t)&,X_\beta(u)]
=
\\
&{}\cdot
X_\alpha(t)X_{3\sigma+\lambda}(u)X_{\alpha}(t)^{-1}
\cdot
X_\alpha(t)X_{3\sigma+2\lambda}(-u^2)X_{\alpha}(t)^{-1}
\\
&{}\cdot
[X_\alpha(t)X_\lambda(u)X_{\alpha}(t)^{-1},
X_\alpha(t)X_\sigma(1)X_{\alpha}(t)^{-1}]
\\
&{}\cdot
X_\alpha(t)X_{2\sigma+\lambda}(u)X_{\alpha}(t)^{-1}
\\
&{}\cdot
X_{2\sigma+\lambda}(-u)
[X_\sigma(1),X_\lambda(u)]
X_{3\sigma+2\lambda}(u^2)
X_{3\sigma+\lambda}(-u).
\end{split}
\end{equation}
Because $\U_\alpha$ centralizes $\U_{3\sigma+\lambda}$,
$\U_{3\sigma+2\lambda}$ and $\U_\lambda$, we may cancel the
$X_\alpha(t)$'s in the first two terms, and in the first term of the
first commutator.  Becase $\U_{3\sigma+2\lambda}$ centralizes all
terms present, we may cancel the terms $X_{3\sigma+2\lambda}(\pm
u^2)$.  The terms between the commutators assemble themselves into
$[X_\alpha(t),X_{2\sigma+\lambda}(u)]$, which equals
$X_{\alpha+2\sigma+\lambda}(3tu)$ by \eqref{eq-Chevalley-relations-for-a-and-2s+l-for-G2-short-root-case}.  Because
$\U_{\alpha+2\sigma+\lambda}$ centralizes all terms present, we may
move this term to the very beginning.  Finally, from \eqref{eq-Chevalley-relations-for-a-and-s-for-G2-short-root-case} one can rewrite
the second terms of the first commutator as
$$
X_\alpha(t)X_\sigma(1)X_\alpha(t)^{-1}
=
X_{2\alpha+\sigma}(3t^2)
X_{\alpha+2\sigma}(3t)
X_{\alpha+\sigma}(2t)
X_\sigma(1).
$$
After all these simplifications, \eqref{eq-G2-big-mess-for-commutator-1} reduces to 
\begin{equation}
\label{eq-G2-big-mess-for-commutator-2}
\begin{split}
[X_\alpha(t)&,X_\beta(u)]
=
X_{\alpha+2\sigma+\lambda}(3tu)
X_{3\sigma+\lambda}(u)
\\
&{}\cdot
[X_\lambda(u),
X_{2\alpha+\sigma}(3t^2)
X_{\alpha+2\sigma}(3t)
X_{\alpha+\sigma}(2t)
X_\sigma(1)]
\\
&{}\cdot
[X_\sigma(1),X_\lambda(u)]
X_{3\sigma+\lambda}(-u).
\end{split}
\end{equation}

Now we focus on the first commutator $[\cdots,\cdots]$.  All its terms
commute with $\U_{2\alpha+\sigma}$, so we may drop the
$X_{2\alpha+\sigma}(3t^2)$ term.  Writing out what remains gives
\begin{equation*}
\begin{split}
[\cdots,\cdots]
={}&
X_\lambda(u)
X_{\alpha+2\sigma}(3t)
X_{\alpha+\sigma}(2t)
X_\sigma(1)
\\
&{}\cdot
X_\lambda(-u)
X_\sigma(-1)
X_{\alpha+\sigma}(-2t)
X_{\alpha+2\sigma}(-3t).
\end{split}
\end{equation*}
By repeatedly using \eqref{eq-Chevalley-relations-for-s-and-a+s-for-G2-short-root-case}--\eqref{eq-Chevalley-relations-for-l-and-a+2s-for-G2-short-roots-case} and the commutativity of
various pairs of root groups, we move all the $X_\lambda$ and
$X_\sigma$ terms to the far right.  A page-long computation yields
$$
[\cdots,\cdots]
=
X_{\alpha+2\sigma+\lambda}(3\e' t u  -6\e\e' t u)
[X_\lambda(u),X_\sigma(1)].
$$

Plugging this into \eqref{eq-G2-big-mess-for-commutator-2}, and
cancelling the commutators and the $X_{3\sigma+\lambda}(\pm u)$ terms, yields
\begin{align*}
[X_\alpha(t),X_\beta(u)]
&{}=
X_{\alpha+2\sigma+\lambda}(3t u+3\e' t u-6\e\e' t u)
\\
&{}=
X_{\alpha+2\sigma+\lambda}(Ctu)
\end{align*}
where $C=0$, $\pm6$ or $12$ depending on $\e,\e'\in\{\pm1\}$.

If $C=0$ (i.e., $\e=\e'=1$) then we have established the
desired Chevalley relation $[\U_\alpha,\U_\beta]=1$ and the proof is complete.
Otherwise we pass to the quotient $\St$ of $\PSt$.  Here $\U_\alpha$ and
$\U_\beta$ commute, so we derive the relation
$X_{\alpha+2\sigma+\lambda}(C t)=1$ in
$\St$.  Since this identity holds universally, it holds for $R=\C$, so
the image of $\U_{\alpha+2\sigma+\lambda}(\C)$ in $\St(\C)$ is the trivial group.
This is a contradiction, since $\St(\C)$ acts on the Kac-Moody algebra
$\g$, with
$X_{\alpha+2\sigma+\lambda}(t)$ acting (nontrivially for $t\neq0$) by $\exp\ad(t e_{\alpha+2\sigma+\lambda})$.
Since $C\neq0$ leads to a contradiction, we must have $C=0$ and so the
Chevalley relation $[\U_\alpha,\U_\beta]=1$ holds in $\PSt$.
%
\end{proof}

\begin{proof}[Case~\XX{5} of theorem~\ref{thm-presentations}]
{\it Assume $\betabar=2\alphabar$ in $\Phibar=BC_{n\geq2}$.}
Choose $\mubar,\lambdabar\in\Phibar$ as shown, and lift them to $\mu,\lambda\in\Phi$
with $2\mu+\lambda=\beta$. (Mnemonic: $\mu$ is middling and $\lambda$
is long.)
$$ 
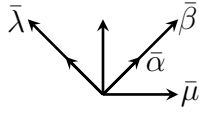
\begin{tikzpicture}[>=stealth, line width=1pt, scale=1]
\draw[->] (0,0) -- (-1,1) node {\llap{$\lambdabar$}\,};
\draw[->] (0,0) -- (-.5,.5);
\draw[->] (0,0) -- (0,1);
\draw[->] (0,0) -- (.5,.5) node {\lower11pt\rlap{\,$\alphabar$}};
\draw[->] (0,0) -- (1,1) node {\rlap{$\betabar$}};
\draw[->] (0,0) -- (1,0) node {\,\rlap{$\mubar$}};
\end{tikzpicture}
$$
As in the case~\XX{2} (when $\alphabar$ and $\betabar$ were the same long root of $\Phibar=B_n$), we can
express any element of $\U_\beta$ in the form
$$
(\hbox{some }x_{\mu+\lambda}\in\U_{\mu+\lambda})
\cdot
\bigl[
(\hbox{some }x_\lambda\in\U_\lambda)
,
(\hbox{some }x_\mu\in\U_\mu)
\bigr]
.
$$
The Chevalley relations in $\PSt$ include the commutativity of
$\U_{\mu+\lambda}$
with $\U_\lambda$, $\U_\mu$ and 
$\U_\alpha$.  So $\U_\alpha$ also centralizes
$\U_\beta$. 
\end{proof}

\begin{proof}[Case~\XX{6} of theorem~\ref{thm-presentations}.]
{\it Assume $\alphabar=\betabar$ is a short root of $\Phibar=BC_{n\geq2}$ and
  $\alpha+\beta$ is a root.}
This is the exceptional case of lemma~\ref{lem-prenilpotent-but-not-classically-prenilpotent-pairs}, and the Chevalley
relation we must establish is not $[\U_\alpha,\U_\beta]=1$.  We will
determine the correct relation during the proof.
We begin by
choosing $\mubar,\sigmabar\in\Phibar$ as shown and lifting them to $\mu,\sigma\in\Phi$
with $\mu+\sigma=\beta$, so $\sigma,\mu$ generate a $B_2$ root
system. 
$$ 
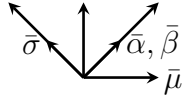
\begin{tikzpicture}[>=stealth, line width=1pt, scale=1]
\draw[->] (0,0) -- (-1,1);
\draw[->] (0,0) -- (-.5,.5) node {\lower11pt\llap{$\sigmabar$\,}};
\draw[->] (0,0) -- (0,1);
\draw[->] (0,0) -- (.5,.5) node {\lower12pt\rlap{\,$\alphabar,\betabar$}};
\draw[->] (0,0) -- (1,1);
\draw[->] (0,0) -- (1,0) node {\lower10pt\rlap{\,$\mubar$}};
\end{tikzpicture}
$$ We choose a generator $e_\gamma$ for the root space of each nonnegative
linear combination $\gamma\in\Phi$ of $\alpha,\sigma,\mu$.  By changing
the signs of $e_{\sigma+\mu}$ and $e_{2\sigma+\mu}$ if necessary, we may suppose that the
Chevalley relations 
\eqref{eq-4-Chevalley-distant-short-and-long}
for
$\sigma$, $\mu$ are
\begin{equation}
\label{eq-Chevalley-relation-for-s-and-m-BCn-case}
[
X_\sigma(t)
,
X_\mu(u)
]
=
X_{\sigma+\mu}(-t u)
X_{2\sigma+\mu}(t^2u),
\end{equation}
Since $\sigma+\mu=\beta$ we may take $t=1$ in
\eqref{eq-Chevalley-relation-for-s-and-m-BCn-case} to express any
element of $\U_\beta$:
\begin{equation}
\label{eq-expression-for-X-b-of-u-BCn-case}
X_\beta(u)
=
X_{2\sigma+\mu}(u)
[
X_\mu(u)
,
X_\sigma(1)
].
\end{equation}
Using this one can express any generator for $[\U_\alpha,\U_\beta]$:
\begin{equation}
\label{eq-expression-for-any-generator-in-B-C-n-coincident-long-root-case}
\begin{split}
[
X_\alpha(t)
,
X_\beta(u)
]
={}&
X_\alpha(t)
X_{2\sigma+\mu}(u)
X_\alpha(t)^{-1}
\\
&{}\cdot
\bigl[
X_\alpha(t)
X_\mu(u)
X_\alpha(t)^{-1}
,
X_\alpha(t)
X_\sigma(1)
X_\alpha(t)^{-1}
\bigr]
\\
&{}\cdot
[
X_\sigma(1)
,
X_\mu(u)
]
\cdot
X_{2\sigma+\mu}(-u).
\end{split}
\end{equation}
By the Chevalley relations $[\U_\alpha,\U_{2\sigma+\mu}]=[\U_\alpha,\U_\mu]=1$, the
$X_\alpha(t)^{\pm1}$'s cancel in the first term and in the first term of
the first commutator.  

Now we consider the Chevalley relations of $\alpha$ and $\sigma$.  Since
$\alphabar+\sigmabar$ is a middling root of $\Phibar$, and $\Phi$ contains every
element of $\Lambda$ lying over every such root, we see that
$\alpha+\sigma$ is a root of $\Phi$.  In particular,
$(\Q\alpha\oplus\Q\sigma)\cap\Phi$ is a $B_2$ root system, in which $\alpha$
and $\sigma$ are orthogonal short roots.  
The Chevalley relations 
\eqref{eq-4-Chevalley-orthogonal-short}
for
$\alpha,\sigma$ are therefore
\begin{equation}
\label{eq-Chevalley-relation-for-a-and-s-BCn-case}
[
X_\alpha(t)
,
X_\sigma(u)
]
=
X_{\alpha+\sigma}(-2t u),
\end{equation}
after changing the sign of $e_{\alpha+\sigma}$ if necessary.

Next, $\mu+\sigma+\alpha=\alpha+\beta$ is a root by hypothesis.  We choose $e_{\mu+\sigma+\alpha}$ so that the
Chevalley relations
\eqref{eq-4-Chevalley-orthogonal-short}
for $\mu$, $\alpha+\sigma$ are
\begin{equation}
\label{eq-Chevalley-relation-for-m-and-a+s-BCn-case}
[
X_\mu(t)
,
X_{\alpha+\sigma}(u)
]
=
X_{\mu+\alpha+\sigma}(-2t u).
\end{equation}

Now we rewrite
\eqref{eq-expression-for-any-generator-in-B-C-n-coincident-long-root-case},
applying the cancellations mentioned above and  rewriting the second term in the
first commutator using \eqref{eq-Chevalley-relation-for-a-and-s-BCn-case}:
\begin{equation}
\label{eq-intermediate-form-of-X-a-bracket-X-b-BCn-case}
\begin{split}
[
X_\alpha(t)
,
X_\beta(u)
]
={}&
X_{2\sigma+\mu}(u)
\cdot
\bigl[
X_\mu(u)
,
X_{\alpha+\sigma}(-2t)
X_\sigma(1)
\bigr]
\\
&{}\cdot
[
X_\sigma(1)
,
X_\mu(u)
]
\cdot
X_{2\sigma+\mu}(-u).
\end{split}
\end{equation}
Now we restrict attention to the first commutator on the right side
and use the Chevalley relations $[\U_{\alpha+\sigma},\U_\sigma]=1$ and
\eqref{eq-Chevalley-relation-for-m-and-a+s-BCn-case} to obtain
\begin{align*}
\bigl[
X_\mu(u)
,
X_{\alpha+\sigma}(-2t)
X_\sigma(1)
\bigr]
={}&
X_\mu(u)
X_{\alpha+\sigma}(-2t)
\cdot
X_\sigma(1)
\\
&{}\cdot
X_\mu(-u)
X_\sigma(-1)
X_{\alpha+\sigma}(2t)
\\
{}={}&
X_{\mu+\alpha+\sigma}(4t u)
X_{\alpha+\sigma}(-2t)
X_\mu(u)
\cdot
X_\sigma(1)
\\
&{}\cdot
X_{\mu+\alpha+\sigma}(4t u)
X_{\alpha+\sigma}(2t)
X_\mu(-u)
X_\sigma(-1).
\end{align*}
The projections to $\Phibar$ of any two roots occurring as subscripts
are linearly independent.  Therefore any two of them are classically
prenilpotent, so their Chevalley relations are present in $\PSt$.
In particular,
$\U_{\mu+\alpha+\sigma}$ centralizes all the other terms;  we gather the $X_{\mu+\alpha+\sigma}(4t u)$
terms at the beginning.
Next, $[\U_\sigma,\U_{\alpha+\sigma}]=1$, so we may move $X_\sigma(1)$ to the right across
$X_{\alpha+\sigma}(2t)$.
Then we can use \eqref{eq-Chevalley-relation-for-m-and-a+s-BCn-case}
again to move $X_\mu(u)$  rightward across $X_{\alpha+\sigma}(2t)$.  The result is
\begin{align*}
\bigl[
X_\mu(u)
,
X_{\alpha+\sigma}(-2t)
X_\sigma(1)
\bigr]
{}={}&
X_{\mu+\alpha+\sigma}(4t u)
[X_\mu(u),X_\sigma(1)]
.
\end{align*}

Plugging this into \eqref{eq-intermediate-form-of-X-a-bracket-X-b-BCn-case} and canceling the commutators gives
\begin{align*}
[
X_\alpha(t)
,
X_\beta(u)
]
={}&
X_{2\sigma+\mu}(u)
X_{\mu+\alpha+\sigma}(4t u)
X_{2\sigma+\mu}(-u)
\notag
\\
={}&
X_{\alpha+\beta}(4t u).
\end{align*}
Tits' Chevalley relation in his definition of $\St$ has the same form,
with the factor $4$ replaced by some integer~$C$.  (Although we don't
need it, we remark that $C=\pm4$ by the second displayed equation in
\cite[\S3.5]{Tits}, or from \cite[Thm.~2(2)]{Morita-root-strings}.
This is related to the fact that $(\Q\alpha\oplus\Q\beta)\cap\Phi$ is a
rank~$1$ affine root system, of type $\BCtilde{}_1^\odd$.)
If $C\neq4$ then in $\St$ we deduce $X_{\alpha+\beta}\bigl((C-4)t
u\bigr)=1$ for all $t,u\in R$ and all rings $R$, leading to the same
contradiction we found in case~\XX{4}.  Therefore $C=4$ and we have
established that Tits' relation already holds in $\PSt$.
\end{proof}

\begin{proof}[Case~\XX{7} of theorem~\ref{thm-presentations}]
{\it Assume $\alphabar=\betabar$ is a short root of $\Phibar=BC_{n\geq2}$ and
  $\alpha+\beta$ is not a root.}
This is similar to the previous case but much easier.  We choose $\mu$,
$\sigma$ and the $e_\gamma$ in the same way, except that $\mu+\sigma+\alpha$ is no longer a
root, so the Chevalley relation
\eqref{eq-Chevalley-relation-for-m-and-a+s-BCn-case} is replaced by
$[\U_\mu,\U_{\alpha+\sigma}]=1$.  We expand $X_\beta(u)$ as in \eqref{eq-expression-for-X-b-of-u-BCn-case} and obtain
\eqref{eq-intermediate-form-of-X-a-bracket-X-b-BCn-case} as before.  But this time the $X_{\alpha+\sigma}(-2t)$ term centralizes
both $\U_\mu$ and $\U_\sigma$, so it vanishes from the commutator.  The right
side of \eqref{eq-intermediate-form-of-X-a-bracket-X-b-BCn-case} then collapses to $1$ and we have proven
$[\U_\alpha,\U_\beta]=1$ in $\PSt$.
\end{proof}

\section{Finite presentations}
\label{sec-finite-presentations}

\noindent
In this section we prove theorem
\ref{thm-finite-presentation-of-Steinberg}, that various
Steinberg and Kac-Moody groups are finitely presented.  At the end we
make several remarks about possible variations on the definition of
Kac-Moody groups.

\begin{proof}[Proof of theorem~\ref{thm-finite-presentation-of-Steinberg}]
We must show that $\St_A(R)$ is finitely presented under either of the
two stated hypotheses.  By theorem~\ref{thm-presentations} it suffices
to prove this with $\PSt$ in place of $\St$.
 
\ref{item-rk-3-and-module-finite-over-suitable-subring} We are
assuming $\rk A=3$ and that $R$ is finitely generated as a module
over a subring generated by finitely many units.
Theorem~\TheoremOnFinitePresentationGeneralTools\PartOfThatTheoremDealingWithTwoSpherical\ of
\cite{Allcock} shows that if $R$ satisfies this hypothesis and $A$ is
$2$-spherical, then $\PSt_A(R)$ is finitely presented.  This proves
\ref{item-rk-3-and-module-finite-over-suitable-subring}.

\ref{item-rk-4-and-finitely-generated-ring} Now we are assuming $\rk A>3$ and that $R$ is finitely generated
as a ring.
Theorem~\TheoremOnFinitePresentationGeneralTools\PartOfThatTheoremDealingWithFGRings\  of \cite{Allcock} gives the finite
presentability of $\PSt_A(R)$ if every pair of nodes of the Dynkin
diagram lies in some irreducible spherical diagram of rank${}\geq3$.
(This use of a covering of $A$ by spherical diagrams was also used by
Capdeboscq \cite{Capdeboscq}.)
By inspecting the list of affine Dynkin diagrams of rank${}>3$, one
checks that this treats all cases of \ref{item-rk-4-and-finitely-generated-ring} except
$$
A=
\begin{tikzpicture}[baseline=-2.5pt, line width=1pt, scale=1]
\draw (0,-.04) -- (1,-.04);
\draw (0,.04) -- (1,.04);
\draw (1,0) -- (1.53,0);
\draw[dashed] (1.53,0) -- (2.47,0);
\draw (2.47,0) -- (3,0);
\draw (3,-.04) -- (4,-.04);
\draw (3,.04) -- (4,.04);
\fill (0,0) circle (.1) node {\raise15pt\hbox{\smash{$\alpha$}}};
\fill (1,0) circle (.1) node {\raise15pt\hbox{\smash{$\beta$}}};
\fill (3,0) circle (.1) node {\raise15pt\hbox{\smash{$\gamma$}}};
\fill (4,0) circle (.1) node {\raise15pt\hbox{\smash{$\delta$}}};
\end{tikzpicture}
$$
(with some orientations of the double edges).  In this case, no
irreducible spherical diagram contains $\alpha$ and $\delta$. 

For this case we use a variation on the proof of 
theorem~\TheoremOnFinitePresentationGeneralTools\PartOfThatTheoremDealingWithFGRings\  of
\cite{Allcock}.
Consider the direct limit $G$ of the groups $\St_B(R)$ as $B$ varies over
all irreducible spherical diagrams of rank${}\geq2$.  If $\rk B\geq3$
then $\St_B(R)$ is finitely presented by theorem~I of Splitthoff
\cite{Splitthoff}.  If $\rk B=2$ then $\St_B(R)$ is finitely
generated by \cite[Lemma\ \LemmaOnFiniteGenerationOfRankTwoSteinbergGroups]{Allcock}.  Since every irreducible rank~$2$
diagram lies in one of rank${}>2$, it follows that $G$ is finitely
presented.  Now, $G$ satisfies all the relations of $\St_A(R)$ except
for the commutativity of $\St_{\singletonalpha}$ with $\St_{\{\delta\}}$.  Because these
groups may not be finitely generated, we might need infinitely many
additional relations to impose commutativity
in the obvious way.

So we proceed indirectly.  Let $Y_\alpha$ be a finite subset of
$\St_{\singletonalpha}$ which together with $\St_{\{\beta\}}$ generates
$\St_{\{\alpha,\beta\}}$.  This is possible since
$\St_{\{\alpha,\beta\}}$ is finitely generated.  We define $Y_\delta$
similarly, with $\gamma$ in place of $\beta$.  We define $H$ as the
quotient of $G$ by the finitely many relations
$[Y_\alpha,Y_\delta]=1$, and claim that the images in $H$ of
$\St_{\singletonalpha}$ and $\St_{\{\delta\}}$ commute.

The following computation in $H$ establishes this.
First, every element of $Y_\delta$ centralizes
$\St_{\{\beta\}}$ by the definition of $G$, and every element of $Y_\alpha$
by definition of $H$.  Therefore it centralizes $\St_{\{\alpha,\beta\}}$,
hence $\St_{\singletonalpha}$.  We've shown that $\St_{\singletonalpha}$ centralizes 
 $Y_\delta$, and it centralizes $\St_{\{\gamma\}}$ by the definition of $G$.
Therefore it centralizes $\St_{\{\gamma,\delta\}}$, hence $\St_{\{\delta\}}$.

$H$ has the same generators as $\PSt_A(R)$, and its defining relations
are among those defining $\PSt_A(R)$.  On the other hand, we have
shown that the generators of $H$ satisfy all the relations in
$\PSt_A(R)$.  So $H\iso\PSt_A(R)$.  In particular, $\PSt_A(R)$ is
finitely presented.

It remains to prove the finite presentability of $\G_{\!A}(R)$ under
the extra hypothesis that the unit group of $R$ is finitely generated
as an abelian group.  This follows  from \cite[Lemma
  \CorollaryOnFiniteGenerationOfKernel]{Allcock}, which says that the quotient of $\PSt_A(R)$ by all
the relations \eqref{eq-torus-relations} is equally well defined by finitely many of
them.  Choosing finitely many such relations, and imposing them on
the quotient $\St_A(R)$ of $\PSt_A(R)$, gives all the relations
\eqref{eq-torus-relations}.  The quotient of $\St_A(R)$ by these is
the definition of
$\G_{\!A}(R)$, proving its finite presentation.
\end{proof}



\begin{remark}[Completions]
We have worked with the ``minimal'' or ``algebraic'' forms of
Kac-Moody groups.  One can consider various completions of it, such as
those surveyed in \cite{Tits-survey}.  None of these completions can
possibly be finitely presented, so no analogue of theorem
\ref{thm-finite-presentation-of-Steinberg} exists.  But it is
reasonable to hope for an analogue
of corollary~\ref{cor-Steinberg-as-direct-limit}.
\end{remark}

\begin{remark}[Chevalley-Demazure group schemes]
If $A$ is spherical then we write $\CD_{\!A}$ for the associated
Chevalley-Demazure group scheme, say the simply-connected version.
This is the unique most natural (in a certain technical sense)
algebraic group over $\Z$ of type~$A$.  If $R$ is a Dedekind domain of
arithmetic type, then the question of whether $\CD_{\!A}(R)$ is finitely
presented was settled by Behr
\cite{Behr-number-field-case}\cite{Behr-function-field-case}.  We
emphasize that our theorem~\ref{thm-finite-presentation-of-Steinberg}
does  not give a new proof of his results, because $\CD_{\!A}(R)$ may be
a proper quotient of $\G_{\!A}(R)$.  The kernel of
$\St_{A}(R)\to\CD_{\!A}(R)$ is called $K_2(A;R)$ and contains the
relators \eqref{eq-torus-relations}.  It can be extremely complicated.

For a non-spherical Dynkin diagram~$A$, the functor $\CD_{\!A}$ is not
defined.  The question of whether there is a good definition, and what
it would be, seems to be completely open.  
Only when $R$ is a field is there known to be a unique ``best''
definition of a Kac-Moody group \cite[theorem~1$'$, p.~553]{Tits}.
The main problem would be
to specify what extra relations to impose on $\G_{\!A}(R)$.  The
remarks below discuss the possible forms of some additional relations.
\end{remark}

\begin{remark}[Kac-Moody groups over integral domains]
If $R$ is an integral domain with fraction field $k$, then it is open
whether $\G_{\!A}(R)\to\G_{\!A}(k)$ is injective.  If $\G_{\!A}$
satisfies Tits' axioms then this would follow from (KMG4), but Tits
does not assert that $\G_{\!A}$ satisfies his axioms.    If
$\G_{\!A}(R)\to\G_{\!A}(k)$ is not injective, then the image seems
better candidate than $\G_{\!A}(R)$ itself,  for the role of ``the''
Kac-Moody group.
\end{remark}

\begin{remark}[Kac-Moody groups via representations]
Fix a root datum~$D$ and a commutative ring~$R$.  By using Kostant's $\Z$-form
of the universal enveloping algebra of $\g$, one can construct a
$\Z$-form $V^\lambda_\Z$ of any integrable highest-weight module
$V^\lambda$ of~$\g$.  Then one defines $V^\lambda_R$ as
$V^\lambda_\Z\tensor R$.  For each real root $\alpha$, one can
exponentiate $\g_{\alpha,\Z}\tensor R\iso R$ to get an action of
$\U_\alpha\iso R$ on $V^\lambda_R$.  One can define the action of the
torus $(\Runits)^n$ directly.  Then one can take the group $\G^\lambda_{\!D}(R)$
generated by these transformations and call it a Kac-Moody group.
This approach is extremely natural and not yet fully worked out.  The
first such work for Kac-Moody groups over rings is Garland's landmark
paper \cite{Garland} treating affine groups; see also Tits' survey \cite[\S5]{Tits-survey},
its references, and the recent articles \cite{Bao-Carbone}\cite{Carbone-Garland}.

Tits \cite[p.~554]{Tits} asserts that this construction allows one to
build a Kac-Moody functor satisfying all his axioms (KMG1)--(KMG9).
We imagine that he reasoned as follows.  First, show that each
$\G^\lambda_{\!D}$ is a Kac-Moody functor and therefore by Tits'
theorem admits a canonical functorial homomorphism from $\G_{\!A}$,
where $A$ is the generalized Cartan matrix of $D$.  (One cannot
directly apply Tits' theorem, because $\G^\lambda_{\!D}(R)$ only comes equipped
with the homomorphisms $\SL_2(R)\to\G^\lambda_{\!D}(R)$ required by
Tits when $\SL_2(R)$ is generated by its subgroups
$\bigl(\begin{smallmatrix}1&*\\0&1\end{smallmatrix}\bigr)$ and
  $\bigl(\begin{smallmatrix}1&0\\ *&1\end{smallmatrix}\bigr)$.
    Presumably this difficulty can be overcome.)  Second, define $I$ as the intersection of the
    kernels of all the homomorphisms $\G_{\!A}\to\G^\lambda_{\!D}$,
    and then define the desired Kac-Moody functor as $\G_{\!A}/I$.
    (This also does not quite make sense, since $\G_{\!A}$ may also
    lack the required homomorphisms from $\SL_2$.  As before,
    presumably this difficulty can be overcome.)
\end{remark}

\begin{remark}[Loop groups]
Suppose $X$ is one of the $ABCDEFG$ diagrams, $\Xtilde$ is its affine
extension as in section~\ref{sec-new-nomenclature}, and $R$ is a commutative ring.  The
well-known description of affine Kac-Moody algebras and loop groups makes it natural to expect that
$\G_{\!\Xtilde}(R)$ is a central extension of $\G_{\!X}(R[t^{\pm1}])$ by~$\Runits$.
The most general results along these lines that I know of are
Garland's theorems 10.1 and B.1 in \cite{Garland}, although they
concern slightly different groups.  Instead, one might simply {\it define}
the loop group $G_{\Xtilde}(R)$ as a central extension of
$\CD_{\!X}(R[t^{\pm1}])$ by $\Runits$, where the $2$-cocycle defining the
extension would have to be made explicit.  Then  one could try to show that $G_{\Xtilde}$
satisfies Tits' axioms.

It is natural to ask whether
such a group $G_{\Xtilde}(R)$ would be finitely presented if $R$ is
finitely generated.  If $\Runits$ is
finitely generated then this is equivalent to the finite presentation
of the quotient $\CD_{\!X}(R[t^{\pm1}])$.  If $\rk X\geq3$ then
$\St_X(R[t^{\pm1}])$ is finitely presented by
Splitthoff's theorem~I of \cite{Splitthoff}.  Then, as Splitthoff
explains in \cite[\S7]{Splitthoff}, the finite presentability of
$\CD_{\!X}(R[t^{\pm1}])$ boils down to properties of $K_1(X,R[t^{\pm1}])$
and $K_2(X,R[t^{\pm1}])$.
\end{remark}

\end{document}